\documentclass[11pt]{amsart}

\usepackage{amssymb}
\usepackage{amsmath}
\usepackage{amsfonts}
\usepackage{graphicx}
\usepackage{amsthm}
\usepackage{enumerate}
\usepackage[mathscr]{eucal}
\usepackage{mathrsfs}
\usepackage{verbatim}
\usepackage{yhmath}

\usepackage{color}

\makeatletter
\@namedef{subjclassname@2010}{%
  \textup{2010} Mathematics Subject Classification}
\makeatother

\numberwithin{equation}{section}
\numberwithin{figure}{section}

\theoremstyle{plain}
\newtheorem{theorem}{Theorem}[section]
\newtheorem{lemma}[theorem]{Lemma}
\newtheorem{proposition}[theorem]{Proposition}

\theoremstyle{plain}

\theoremstyle{remark}
\newtheorem{remark}[theorem]{Remark}

\DeclareMathOperator{\supp}{supp}
\DeclareMathOperator{\xsupp}{x-supp}

\DeclareMathOperator{\sgn}{sgn}
\DeclareMathOperator*{\esssup}{ess\,sup}

\allowdisplaybreaks[4]

\begin{document}

\title
[Bipolynomial Roth theorems]{Two bipolynomial Roth theorems in $\mathbb{R}$}

\author{Xuezhi Chen}
\address{School of Mathematical Sciences\\
University of Science and Technology of China\\
Hefei, 230026\\ P.R. China}
\email{cxz2013@mail.ustc.edu.cn}

\author{Jingwei Guo}
\address{School of Mathematical Sciences\\
University of Science and Technology of China\\
Hefei, 230026\\ P.R. China}
\email{jwguo@ustc.edu.cn}

\author{Xiaochun Li}
\address{Department of Mathematics\\
University of Illinois at Urbana-Champaign\\
Urbana, IL, 61801\\USA}
\email{xcli@math.uiuc.edu}

\date{\today}

\thanks{J. G. is partially supported by the Fundamental Research Funds for the Central Universities (No. WK3470000013) and the NSFC Grant (No. 11571331). X. L. is supported by Simons fellowship in Math 2019-2020. }



\begin{abstract}
We give two Roth theorems, related to the nonlinear configuration $x$, $x+P_1(t)$, $x+P_2(t)$ involving two polynomials, for sets in $\mathbb{R}$ of positive density and of fractional dimensions. The proof uses Fourier analysis.
\end{abstract}

\maketitle


\section{Introduction}\label{sec1}

The polynomial Szemer\'{e}di theorem of Bergelson and Leibman \cite{BL96} asserts that if $P_1$, \ldots, $P_m\in \mathbb{Z}[t]$ all have zero constant term, then any subset $S$ of $[N]:=[1,N]\cap \mathbb{N}$ containing no nontrivial polynomial progression
\begin{equation*}
x, x+P_1(t), \ldots, x+P_m(t)
\end{equation*}
satisfies $|S|=o(N)$. This gives a qualitative description on size of $S$ which contains the polynomial progression mentioned above. It becomes much more challenging and interesting to find quantitative bounds of $|S|$. Among many partial progresses are a pair of papers \cite{PP19, PP20} by Peluse and Prendiville recently, in which they obtained bounds for subsets of $[N]$ lacking the nonlinear Roth configuration $x$, $x+t$, $x+t^2$.

In the setting of real numbers, Bourgain \cite{Bourgain88} proved that given $\varepsilon>0$ and integer $d\geq 2$ there is a $\delta>0$ such that if $S$ is a measurable set of $[0, N]$, $|S|>\varepsilon N$, then there is a triple $x$, $x+t$, $x+t^d$ in $S$ with $t>\delta N^{1/d}$. Durcik, Guo and Roos \cite{DGR19} further studied nonlinear patterns $x, x+t, x+P(t)$ for a monic polynomial $P$ with $P(0)=0$ and $\deg(P) \geq 2$. In particular $P$ is allowed to have a linear term and $\delta$ depends only on $\varepsilon$, $\deg (P)$ and an upper bound of the $\ell^1$-sum of coefficients of $P$. Krause \cite{Krause1901} investigated the case when $P$ is a non-flat curve.

Our first result is the following theorem, related to the Roth configuration $x$, $x+P_1(t)$, $x+P_2(t)$ involving polynomials of different degrees,  for sets in $\mathbb{R}$ of positive density.

\begin{theorem}\label{thm1}
Let $P_i : \mathbb{R}\rightarrow \mathbb{R}$
$(i=1,2)$ be two polynomials satisfying $1\leq \deg(P_1)<\deg(P_2)$ and $P_1(0)=P_2(0)=0$. For any $\varepsilon>0$ there exists a $\delta=\delta(\varepsilon, P_1, P_2)$ with
\begin{equation}
\delta\geq \exp\left(-\exp\left(c\varepsilon^{-6}\right)\right)\label{s1-2}
\end{equation}
for some constant $c=c(P_1, P_2)>0$, such that, given any measurable set $S\subset [0,N]$ with $N>1$ and measure $|S|\geq \varepsilon N$, it contains a triplet
\begin{equation*}
x, x+P_1(t), x+P_2(t)
\end{equation*}
with $t>\delta N^{1/\deg(P_2)}$.
\end{theorem}

Let us outline some ideas in the proof of Theorem \ref{thm1}, whose details will be presented in
Section \ref{sec2}--\ref{sec6}.
Firstly we argue similarly as in Bourgain \cite{Bourgain88} and Durcik, Guo and Roos \cite{DGR19} to reduce the problem to proving certain bilinear estimate containing two given polynomials. Since $P_1(t)$ is not simply a linear function any more, we use a substitution to make it linear in order to apply Bourgain's convolution trick, which is something slightly different from previous work.
 See Section \ref{sec2}.

A key step in our proof of the desired bilinear estimate is, by changing variables, to transform the pair of polynomials into (up to errors) the form of $\{c_1\omega, c_2\omega+c_3\omega^{\nu} \}$ for some positive $\nu\neq 1$ (see Section \ref{sec4}).  This novelty transformation turns out to be handy for the bipolynomial case. Surprisingly we will see in Section \ref{sec5} that the influence of errors cannot be ignored in some circumstance. This indicates that the errors cannot be dropped trivially. This phenomenon
does  not occur in the work in \cite{Bourgain88, DGR19}.
Hence, we are forced to seek new technical power to resolve this difficulty.  As a preparation for later sections,  in Section \ref{sec4} we  formulate and establish results  concerning sizes or asymptotics of the kernel $K$, which can be proved by a delicate application of the method of stationary phase and integration by parts.

If $c_2\neq 0$, we deal with the case when two polynomials have the same smallest power. In the special case when $P_1(t)=t$ and $P_2(t)$ has a linear term (corresponding to an integer $\nu\geq 2$), \cite[Section 4.3]{DGR19} took advantage of the nonlinearity of the nonlinear term of $P_2(t)$ and showed that certain key mixed derivative has a small lower bound (which decreases exponentially as $m$ increases) outside finitely many short intervals of length $O(2^{-\gamma m})$. Based on that they established the desired bilinear estimate. A key ingredient in our proof of the bipolynomial case is that we are able to show that the key mixed derivative is always nonvanishing and has a lower bound independent of $m$. This information enables us to give a simple proof of the desired estimate with a clearer and better decay rate, which leads to a better range of $\beta$ in our next theorem on the bipolynomial progress in fractional $\mathbb R$.   In the proof we use H\"ormander's \cite[Theorem 1.1]{Hormander73} and need to be clear about what its implicit constant depends on (see Theorem \eqref{app1-1}). See Section \ref{sec5}.

To prove its lower bound one can first derive a formula of the mixed derivative by a routine computation. If $\nu\neq 2$, although the formula is really long its structure is quite clear---one can separate it into several parts with each part having different magnitude. In particular the error caused by the transformation from the polynomials to a simple form (mentioned above) has no essential influence. Hence one can determine the size of the mixed derivative readily. If $\nu=2$, however, the influence of the error is relatively significant and the above separation is not attainable.
This is a new enemy, never appearing in any known work. To overcome this difficulty we manage to express the mixed derivative as a product of some nonzero factors and a polynomial factor. Once we prove the polynomial is nontrivial, we know it must have a lower bound which is possibly much smaller than $1$.

If $c_2=0$, we follow the strategy developed by the third author in \cite{Li13} to prove desired bilinear estimates as what has been done in \cite{DGR19}.
Besides oscillatory integrals, the proof relies heavily on a concept called $\sigma$-uniformity. This concept was inspired by Gowers' work in \cite{Gowers98} and crucial in \cite{Li13}'s study of bilinear Hilbert transforms along monomials. Such a strategy involving the $\sigma$-uniformity was later often used to study variants of the bilinear Hilbert transform, for example, the polynomial case in \cite{LX16}, the general curve case in \cite{GX16} and the bipolynomial case in \cite{Dong19}. Since we have transformed the bipolynomial to a general curve, our treatment in Section \ref{sec6}  is  similar to those in \cite{Li13, GX16}.


Our second result is the following theorem, related to the Roth configuration $x$, $x+P_1(t)$, $x+P_2(t)$ involving linearly independent polynomials, for sets in $\mathbb{R}$ of fractional dimensions.

\begin{theorem}\label{thm2}
Let $P_i : \mathbb{R}\rightarrow \mathbb{R}$ $(i=1,2)$ be two linearly independent polynomials satisfying $P_1(0)=P_2(0)=0$. Assume that $E\subset[0,1]$ is a closed set which supports a probability measure $\mu$ satisfying
\begin{itemize}
\item[(A)]$\mu([x,x+\epsilon])\leq C_1\epsilon^\alpha$ for all $0<\epsilon\leq1$,
\item[(B)]$|\widehat{\mu}(k)|\leq C_2(1-\alpha)^{-B}|k|^{-\frac{\beta}{2}}$ for all $k\in \mathbb{Z}\setminus \{0\}$,
\end{itemize}
where $0<\alpha<1$ and $8/9<\beta\leq 1$. If $\alpha>1-\epsilon_0$ for a sufficiently small constant $\epsilon_0>0$ depending only on $P_1$, $P_2$, $C_1$, $C_2$, $B$ and $\beta$, then $E$ contains a triplet
\begin{equation*}
x, x+P_1(t), x+P_2(t)
\end{equation*}
for some $t>0$.
\end{theorem}

{\L}aba and Pramanik \cite{LP09} first showed that, under the dimensionality and Fourier decay assumptions (A) and (B) with $2/3<\beta\leq 1$, if $\alpha$ is sufficiently close to $1$ then the set $E$ contains an arithmetic progression $x$, $x+t$, $x+2t$ for some $t>0$. Recently Fraser, Guo and Pramanik \cite{FGP19} proved that, under the same assumptions (A) and (B) with $1-s_0<\beta<1$ for certain constant $s_0$ depending only on a given polynomial $P$ with $\deg (P)\geq 2$ and $P(0)=0$, if $\alpha$ is sufficiently close to $1$ then the set $E$ contains a nonlinear configuration $x$, $x+t$, $x+P(t)$ for some $t>0$. Krause \cite{Krause1904} studied the same polynomial configuration problem with an emphasis on assuming only sufficiently large Hausdorff dimension.

For other interesting related results, especially in high dimensions, see for example Chan, {\L}aba and Pramanik \cite{CLP16}, Henriot, {\L}aba and Pramanik \cite{HLI16} and references therein.

Our result extends the main result in \cite{FGP19} to a bipolynomial setting with $8/9<\beta\leq 1$.
The estimates with a better decay rate obtained in Case \ref{s5-subcase1} and \ref{s5-subcase2} (see Section \ref{sec5} and Subsection \ref{sec6-1}) allow us to strengthen the implicit value $s_0$ in \cite{FGP19} to an explicit value $1/9$, which leads to a better range of $\beta$. One key point in the proof of Theorem \ref{thm2} is to prove a generalized Sobolev improving estimate involving information of two polynomials.   Another interesting point is that we only assume that two polynomials are linearly independent rather than of different degrees since the existence of the desired Roth configuration is only proved for some small $t$.


\medskip

{\it Notations.} For real $X$ and nonnegative $Y$, we use $X\lesssim Y$ to denote $|X|\leq CY$ for some constant $C$. We write $X\lesssim_p Y$ to indicate that the implicit constant $C$ depends on a parameter $p$. If $X$ is nonnegative, $X\gtrsim Y$ means $Y\lesssim X$. The Landau notation $X=O_p(Y)$ is equivalent to $X\lesssim_p Y$. The notation $X\asymp Y$ means that $X\lesssim Y$ and $Y\lesssim X$. We let $\mathbb{Z}_+=\mathbb{N}\cup \{0\}$ and $e(x)=\exp(2\pi i x)$. The Fourier transform of $f$ is $\widehat{f}(\xi)=\mathcal{F}(f)(\xi)=\int_{\mathbb{R}} \! f(x)e(-\xi x) \,\textrm{d}x$. $a \gg (\ll)$ $b$ means $a$ is much greater (less) than $b$.  $\textbf{1}_{E}$ represents the characteristic function of a set $E$.


\section{Reduction of Theorem \ref{thm1}}\label{sec2}

Throughout this paper we denote two real polynomials by
\begin{equation}
P_1(t)=a_{\sigma_1}t^{\sigma_1}+a_{\sigma_1+1}t^{\sigma_1+1}+\cdots+a_{d_1}t^{d_1} \label{s2-8}
\end{equation}
and
\begin{equation}	
P_2(t)=b_{\sigma_2}t^{\sigma_2}+b_{\sigma_2+1}t^{\sigma_2+1}+\cdots+b_{d_2}t^{d_2},\label{s2-9}
\end{equation}
where $a_{\sigma_1}, a_{d_1}$, $b_{\sigma_2}, b_{d_2}$ are nonzero, $1\leq \sigma_1\leq d_1$ and $1\leq \sigma_2\leq d_2$.

For Theorem \ref{thm1} we assume $d_1<d_2$. To prove Theorem \ref{thm1} it suffices to prove that there is a $\delta=\delta(\varepsilon, P_1, P_2)$ with \eqref{s1-2} such that
\begin{equation}\label{lbdd}
\int_0^N \!\!\!\! \int_0^{N^{1/d_2}}\!\!  f(x)f\left(x+P_1(t)\right)f\left(x+P_2(t)\right) \,\textrm{d}t\textrm{d}x>\delta N^{1+\frac{1}{d_2}}
\end{equation}
for all measurable functions $f$ on $\mathbb{R}$ with $\supp(f)\subset [0, N]$, $0\leq f\leq 1$ and $\int_0^N \!\! f\geq \varepsilon N$. Then the desired result follows easily by taking $f=\textbf{1}_{S}$.

Equivalently, by rescaling, it can be reduced to prove that
\begin{equation}
\int_0^1 \!\!\! \int_0^1 \!\!  f(x)f\left(x+N^{-1}P_1\left(N^{\frac{1}{d_2}}t\right)\right)f\left(x+N^{-1}P_2\left(N^{\frac{1}{d_2}}t\right)\right) \,\textrm{d}t\textrm{d}x>\delta\label{s2-7}
\end{equation}
for all measurable functions $f$ on $\mathbb{R}$ with $\supp(f)\subset [0, 1]$, $0\leq f\leq 1$ and $\int_0^1 \! f\geq \varepsilon$.

We follow the approach used in \cite{Bourgain88} and also \cite[Section 2]{DGR19} to reduce the problem to the following bilinear estimate.

Let $\tau$ be a nonnegative smooth bump function supported in $[1/2,2]$ with integral $1$. Let $\tau_l(t)=2^l \tau(2^l t)$. We state our main proposition as follows.

\begin{proposition}\label{s2-prop1}
Let $P_1$ and $P_2$, denoted by \eqref{s2-8} and  \eqref{s2-9}, be such that $d_1<d_2$. There exist $\mathfrak{b}>0$ and sufficiently large $\Gamma_1\in \mathbb{N}$ such that for any $m\in \mathbb{Z}_+$, $j\in \Gamma_1 (2\mathbb{Z}_+)$, $l\in \Gamma_1 (\mathbb{Z}_+ \!\! \setminus \!\! 2\mathbb{Z}_+)$ \footnote{Here $\Gamma_1 (2\mathbb{Z}_+)=\{0, 2\Gamma_1, 4\Gamma_1, \ldots \}$ and $\Gamma_1 (\mathbb{Z}_+ \!\! \setminus \!\! 2\mathbb{Z}_+)=\{\Gamma_1, 3\Gamma_1, 5\Gamma_1,\ldots \}$.}and Schwartz functions $f$ and $g$ with $\supp(\widehat{g})\subset \{\xi\in \mathbb{R}: 2^m\leq |\xi|\leq 2^{m+1}\}$, we have
\begin{equation}
\begin{split}
&\quad\left\|\int \!\! f\left(x+2^{-jd_2}P_1\left(2^{j}t\right)\right)g\left(x+2^{-jd_2}P_2\left(2^{j}t\right)\right) \tau_l(t) \,\textrm{d}t\right\|_{L^1_x \left([0,1] \right)} \\
&\lesssim 2^{\mathfrak{b} l} 2^{-m/16} \left\|f\right\|_2\left\|g\right\|_2.
\end{split}\label{s1-1}
\end{equation}
\end{proposition}

\begin{remark}\label{s2-rm1}
From the proof of Proposition \ref{s2-prop1}, it is easy to observe that the condition $d_1<d_2$ is only used if $j-l\geq \Gamma_1>0$ (see Case \ref{s5-subcase3} below); if $j-l\leq -\Gamma_1<0$ (see Case \ref{s5-subcase1} and \ref{s5-subcase2} below) a weaker assumption that $P_1$ and $P_2$ are linearly independent (instead of $d_1<d_2$) would suffice to yield a better decay factor  $2^{-m/6}$ (instead of $2^{-m/16}$).

This observation will be particularly useful in the proof of Theorem \ref{thm2}.
\end{remark}

Proposition \ref{s2-prop1} will be proved  in Section \ref{sec3}--\ref{sec6}. In the remaining part of this section, let us see why Proposition \ref{s2-prop1} implies \eqref{lbdd}, and therefore Theorem \ref{thm1}.
We first assume $N=2^{jd_2}$, $j\in \Gamma_1 (2\mathbb{Z}_+)$, $\Gamma_1\in \mathbb{N}$. Let $\rho\in C_c^{\infty}(\mathbb{R})$ be nonnegative, even, constant on $[-1,1]$ with $\supp(\rho)\subset[-2,2]$ and $\int \! \rho=1$. Denote $\rho_l(t)=2^l \rho(2^l t)$, $p_{ij}(t)=2^{-jd_2}P_i(2^{j}t)$ and
\begin{equation*}
I=\int_0^1 \!\!\! \int_0^1 \!\!  f(x)f\left(x+p_{1j}(t)\right)f\left(x+p_{2j}(t)\right) \,\textrm{d}t\textrm{d}x.
\end{equation*}

For $l, l', l''\in \Gamma_1 (\mathbb{Z}_+ \!\! \setminus \!\! 2\mathbb{Z}_+)$ with $l'<l<l''$ we have
\begin{equation*}
2^l I\gtrsim_{\tau} \int_0^1 \!\!\! \int_0^1 \!\!  f(x)f\left(x+p_{1j}(t)\right)f\left(x+p_{2j}(t)\right) \tau_l(t) \,\textrm{d}t\textrm{d}x.
\end{equation*}	
The integral on the right side is the sum of
\begin{align*}
I_1&=\int_0^1 \!\!\! \int_0^1 \!\!  f(x)f\left(x+p_{1j}(t)\right)f*\rho_{l'}\left(x+p_{2j}(t)\right) \tau_{l}(t) \,\textrm{d}t\textrm{d}x,\\
I_2&=\int_0^1 \!\!\! \int_0^1 \!\!  f(x)f\left(x+p_{1j}(t)\right)(f*\rho_{l''}-f*\rho_{l'})\left(x+p_{2j}(t)\right) \tau_{l}(t) \,\textrm{d}t\textrm{d}x,\\
I_3&=\int_0^1 \!\!\! \int_0^1 \!\!  f(x)f\left(x+p_{1j}(t)\right)(f-f*\rho_{l''})\left(x+p_{2j}(t)\right) \tau_{l}(t) \,\textrm{d}t\textrm{d}x.
\end{align*}

Following the boundedness of $f$ and H\"older's inequality, it is obvious that
\begin{equation*}
I_2=O\left(\left\|f*\rho_{l''}-f*\rho_{l'}\right\|_2\right).
\end{equation*}

In addition, by a dyadic decomposition on the frequency side and using Proposition \ref{s2-prop1}, we have
\begin{equation*}
I_3=O \left( 2^{\mathfrak{b}l-\frac{1}{17}l''}\right).
\end{equation*}

We set
\begin{equation}
I_1':=\int_0^1 \!\!  f(x)f*\rho_{l'}\left(x\right) \left(\int_0^1 \!\!  f\left(x+p_{1j}(t)\right)\tau_{l}(t) \,\textrm{d}t\right) \,\textrm{d}x\,. \label{s2-2}
\end{equation}
Then by the mean value theorem we have
\begin{equation*}
I_1-I_1'=O_{P_2}\left(2^{l'-l}\right).
\end{equation*}
Notice that the inner integral in \eqref{s2-2} can be represented as
\begin{equation}
\int_0^1 \!\!  f\left(x+p_{1j}(t)\right)\tau_{l}(t) \,\textrm{d}t=\int \!\!  f\left(x+2^{-jd_2}P_1(t)\right)\tau_{l-j}(t) \,\textrm{d}t, \label{s2-1}
\end{equation}
$t\asymp 2^{j-l}$ and $|j-l|\geq \Gamma_1$. If $\Gamma_1$ is sufficiently large, the size of $P_1(t)$ is dominated by its monomial $a_{s}t^{s}$, $s=d_1$ or $\sigma_1$. This leads us to use the substitution
\begin{equation*}
\omega=2^{-jd_2}|P_1(t)|
\end{equation*}
to rewrite \eqref{s2-1} as a convolution. We may assume that $a_s<0$ while the case $a_s>0$ is the same up to a reflection. Therefore
\begin{equation*}
\eqref{s2-1}=f*\widetilde{\tau}(x),
\end{equation*}
where
\begin{equation*}
\widetilde{\tau}(\omega)=\tau_{l-j}(t(\omega))t'(\omega).
\end{equation*}
Let $\varsigma_{s}^{j,l}=|a_s|2^{s(j-l)-jd_2}$ and $\rho_{\varsigma_{s}^{j,l}}(x)=(\varsigma_{s}^{j,l})^{-1}\rho((\varsigma_{s}^{j,l})^{-1}x)$. Then
\begin{align*}
\left\|f*\widetilde{\tau}-f*\rho_{\varsigma_{s}^{j,l'}}\right\|_2
&\leq \left\|f*\rho_{\varsigma_{s}^{j,l''}}-f*\rho_{\varsigma_{s}^{j,l'}}\right\|_2 \\
&\quad +\left\|\widetilde{\tau}-\widetilde{\tau}*\rho_{\varsigma_{s}^{j,l''}}\right\|_1+ \left\|\widetilde{\tau}*\rho_{\varsigma_{s}^{j,l'}}-\rho_{\varsigma_{s}^{j,l'}}\right\|_1\\
&=\left\|f*\rho_{\varsigma_{s}^{j,l''}}-f*\rho_{\varsigma_{s}^{j,l'}}\right\|_2+O(2^{l-l''})+O(2^{l'-l}).
\end{align*}
The last two bounds follow from rescaling and the mean value theorem. We thus readily get
\begin{equation*}
|I_1'-I_1''|\leq \sum_{s=\sigma_1, d_1}\left\|f*\rho_{\varsigma_{s}^{j,l''}}-f*\rho_{\varsigma_{s}^{j,l'}}\right\|_2+O(2^{l-l''})+O(2^{l'-l}),
\end{equation*}
where
\begin{equation}\label{I''1}
I_1'':=\int_0^1 \!\!  f(x)f*\rho_{l'}\left(x\right)f*\rho_{\varsigma_{s}^{j,l'}}\left(x\right) \,\textrm{d}x\,.
\end{equation}
By  a lemma of Bourgain \cite[Lemma 6]{Bourgain88}, we see that $I''_1$ obeys
\begin{equation*}
I_1''\geq c_{\rho} \left(\int_{0}^{1} \! f \right)^3\geq c_{\rho}\varepsilon^3.
\end{equation*}

Collecting the above upper and lower bounds yields that if $l''$ (resp. $l$) is large enough\footnote{This is quantifiable.} with respect to $l$ (resp. $l'$) then
\begin{equation*}
2^l I+\left\|f*\rho_{l''}-f*\rho_{l'}\right\|_2+\sum_{s=\sigma_1, d_1}\left\|f*\rho_{\varsigma_{s}^{j,l''}}-f*\rho_{\varsigma_{s}^{j,l'}}\right\|_2\geq c\varepsilon^3.
\end{equation*}
In fact we can choose from $\Gamma_1 (\mathbb{Z}_+ \!\! \setminus \!\! 2\mathbb{Z}_+)$ a sequence $\Gamma_1=l_1<l_2<\cdots<l_k<\cdots$ (independently of $f$ and $j$) such that for each $k\in\mathbb{N}$ we have $l_{k+1}\asymp (17\mathfrak{b})^{k}\log \varepsilon^{-1}$ and that either
\begin{equation}
I>2^{-l_{k+1}-1}c\varepsilon^3   \label{s2-3}
\end{equation}
or
\begin{equation}
\left\|f*\rho_{l_{k+1}}-f*\rho_{l_{k}}\right\|_2+\sum_{s=\sigma_1, d_1}\left\|f*\rho_{\varsigma_{s}^{j,l_{k+1}}}-f*\rho_{\varsigma_{s}^{j,l_{k}}}\right\|_2\geq c\varepsilon^{3}/2. \label{s2-4}
\end{equation}
Notice that by using the Plancherel theorem and the fast decay of $\hat{\rho}$ we have
\begin{equation*}
\sum_{k=1}^{\infty}\left(\left\|f*\rho_{l_{k+1}}-f*\rho_{l_{k}}\right\|_2^2+\sum_{s=\sigma_1, d_1}\left\|f*\rho_{\varsigma_{s}^{j,l_{k+1}}}-f*\rho_{\varsigma_{s}^{j,l_{k}}}\right\|_2^2\right)\leq C_{\rho}.
\end{equation*}
Hence \eqref{s2-4} can only occur a bounded number of times and \eqref{s2-3} must hold for some $1\leq k_0\leq K:=\lceil 12c^{-2}C_\rho\varepsilon^{-6}\rceil+1$. Then
\begin{equation*}
I>2^{-l_{k_0+1}-1}c\varepsilon^3\geq2^{-l_{K+1}-1}c\varepsilon^3.
\end{equation*}
Using the size estimate of $l_{K+1}$, we can conclude that there exists a $\delta=\delta(\varepsilon,P_1,P_2)$ satisfying \eqref{s1-2} and \eqref{s2-7}.

For general $N>1$ we use the pigeonhole principle to complete the proof. Assume
\begin{equation*}
2^{jd_2}<N<2^{(j+2\Gamma_1)d_2}.
\end{equation*}
for some $j\in\Gamma_1(2\mathbb{Z_+})$. Let $N_0=N2^{-jd_2}$. Then $1<N_0<2^{2\Gamma_1d_2}$. For any measurable $f$ with $\supp(f) \subset [0, N]$, $0\leq f\leq 1$ and $\int_0^N \!\! f\geq \varepsilon N$, we have
\begin{equation*}
  \sum_{i=0}^{\lceil N_0 \rceil-2}\int_{i 2^{jd_2}}^{(i+1) 2^{jd_2}}\!\! f+\int_{(N_0-1)  2^{jd_2}}^{N_0 2^{jd_2}}\!\! f\geq \int_0^N \!\! f \geq \varepsilon N.
\end{equation*}
Hence there exists an $i_0\in \{0,1,\ldots,\lceil N_0\rceil-2,N_0-1\}$ such that
\begin{equation*}
   \int_{i_0 2^{jd_2}}^{(i_0+1) 2^{jd_2}}\!\! f \geq \frac{\varepsilon N}{\lceil N_0\rceil} \geq \frac{\varepsilon}{2}2^{jd_2}.
\end{equation*}
Therefore
\begin{align}
   &\quad \int_0^N \!\!\!\! \int_0^{N^{1/d_2}}\!\!  f(x)f\left(x+P_1(t)\right)f\left(x+P_2(t)\right) \,\textrm{d}t\textrm{d}x\label{s2-5}\\
   &\geq\int_{i_0 2^{jd_2}}^{(i_0+1) 2^{jd_2}} \!\!\!\! \int_0^{2^j}\!\!  f(x)f\left(x+P_1(t)\right)f\left(x+P_2(t)\right) \,\textrm{d}t\textrm{d}x\nonumber\\
   &\geq\int_0^{2^{jd_2}} \!\!\!\! \int_0^{2^j}\!\!  g(x)g\left(x+P_1(t)\right)g\left(x+P_2(t)\right) \,\textrm{d}t\textrm{d}x,\nonumber
\end{align}
where $g(x)=f(x+i_0 2^{jd_2})\textbf{1}_{[0,2^{jd_2}]}(x)$ is measurable such that $\supp(g)\subset[0,2^{jd_2}]$, $0\leq g\leq 1$ and
\begin{equation*}
\int_0^{2^{jd_2}}\!\! g=\int_{i_0 2^{jd_2}}^{(i_0+1) 2^{jd_2}}\!\! f\geq\frac{\varepsilon}{2}2^{jd_2}.
\end{equation*}
Applying to $g$ the conclusion from the first part yields
\begin{equation*}
\eqref{s2-5}>\delta(\varepsilon/2, P_1,P_2)2^{jd_2(1+\frac{1}{d_2})}>\tilde{\delta}N^{1+\frac{1}{d_2}}
\end{equation*}
with $\tilde{\delta}=2^{-2\Gamma_1(d_2+1)}\delta(\varepsilon/2, P_1,P_2)$, which implies \eqref{lbdd}.
Therefore, to establish Theorem  \ref{thm1}, it remains to prove Proposition \ref{s2-prop1}.


\section{Reduction of Proposition \ref{s2-prop1}} \label{sec3}

We give a proof of Proposition \ref{s2-prop1} in Section \ref{sec3}--\ref{sec6}. We divide it into several sections for simplification and clarification.

Let
\begin{equation*}
\widehat{\mathbb{P}_c f}(\xi)=\widehat{f}(\xi)\mathbf{1}_{[2^c, 2^{c+1})}(|\xi|).
\end{equation*}
To prove \eqref{s1-1} we need to estimate the $L^1([0,1])$ norm of
\begin{equation}
\sum_{k\in\mathbb{Z}}  \int \!\! \mathbb{P}_k f\left(x+2^{-jd_2}P_1\left(2^{j}t \right) \right) g\left(x+2^{-jd_2}P_2\left(2^{j}t \right) \right) \tau_l(t) \,\textrm{d}t \label{s4-1}
\end{equation}
which is, by the Fourier inversion,
\begin{equation*}
\sum_{k\in\mathbb{Z}} \iint \widehat{\mathbb{P}_k f}(\xi) \widehat{g}(\eta) e\left( (\xi+\eta)x \right) \mathfrak{m}_{j,l}(\xi, \eta)  \,\textrm{d}\xi\textrm{d}\eta,
\end{equation*}
where
\begin{equation}
\mathfrak{m}_{j,l}(\xi, \eta)=\int \!\! \tau(t) e\left(2^{-jd_2}\left(\xi P_1\left(2^{j-l}t \right)+\eta P_2\left(2^{j-l}t \right)\right)\right)  \,\mathrm{d}t. \label{s4-2}
\end{equation}

If the sizes of $P_1(2^{j-l}t)$ and $P_2(2^{j-l}t )$ are dominated by their monomials $a_{r} (2^{j-l}t)^{r}$ and  $b_{r'} (2^{j-l}t )^{r'}$ respectively\footnote{Since we only consider those $j$ and $l$ with $|j-l|\geq \Gamma_1$ for a sufficiently large $\Gamma_1$, both polynomials are dominated by their own monomials with the largest or smallest powers.}, then we denote
\begin{equation*}
|a_r|=2^{\alpha_r},\quad |b_{r'}|= 2^{\beta_{r'}},\quad m_0=\beta_{r'}-\alpha_r+(j-l)(r'-r)
\end{equation*}
and rewrite \eqref{s4-1} as
\begin{equation}
\bigg(\sum_{\substack{k\in\mathbb{Z}\\|k|\geq \mathscr{K}}}+\sum_{\substack{k\in\mathbb{Z}\\|k|<\mathscr{K}}}\bigg)\!\! \iint \!\! \mathcal{F}\left(\mathbb{P}_{m+m_0+k} f\right)(\xi) \widehat{g}(\eta) e\left( (\xi+\eta)x \right) \mathfrak{m}_{j,l}(\xi, \eta)  \,\textrm{d}\xi\textrm{d}\eta \label{s4-3}
\end{equation}
for some constant $\mathscr{K}>0$.

Notice that the phase function in $\mathfrak{m}_{j,l}(\xi, \eta)$ has no critical points if $\mathscr{K}$ is greater than some large absolute constant. Integration by parts once yields
\begin{equation*}
\|\mathfrak{m}_{j,l}\|_{\infty}\lesssim_{P_1,P_2} 2^{d_2 l}2^{-m}.
\end{equation*}
By using this bound, the duality of $L^1$ and H\"older's inequality, one can readily show that the $L^1$ norm of the first sum in \eqref{s4-3} over all integers $|k|\geq \mathscr{K}$ is of size
\begin{equation*}
O_{P_1,P_2}\left( 2^{d_2 l}2^{-m/2}\|f\|_2\|g\|_2\right)
\end{equation*}
which is smaller than the desired bound in \eqref{s1-1}. Indeed for $h\in L^{\infty}([0,1])$ we have
\begin{align*}
&\quad \left|\int_0^1  \!\!\!\!
\iint \!\! \sum_{|k|\geq \mathscr{K}}\!\! \mathcal{F}\left(\mathbb{P}_{m+m_0+k} f\right)(\xi) \widehat{g}(\eta) e\left( (\xi+\eta)x \right) \mathfrak{m}_{j,l}(\xi, \eta)h(x)  \,\textrm{d}\xi\textrm{d}\eta\textrm{d}x \right|\\
&\lesssim 2^{d_2 l}2^{-m}\iint \! \left| \widehat{f}(\xi) \widehat{g}(\eta)\widehat{h}(-\xi-\eta)\right|  \,\textrm{d}\xi\textrm{d}\eta\\
&\lesssim 2^{d_2 l}2^{-m/2}\|f\|_2\|g\|_2\|h\|_{\infty},
\end{align*}
as desired.

For the second sum in \eqref{s4-3}, it suffices to prove that for any fixed $|k|<\mathscr{K}$ and some absolute constant $b>0$
\begin{equation*}
\left\| \iint \!\! \widehat{f}(\xi) \widehat{g}(\eta) e\left( (\xi+\eta)x \right) \mathfrak{m}_{j,l}(\xi, \eta)  \,\textrm{d}\xi\textrm{d}\eta\right\|_{L^1_x\left([0,1]\right)} \lesssim_{\mathscr{K}} 2^{bl}2^{-m/16}\|f\|_2\|g\|_2
\end{equation*}
for any $f, g\in \mathscr{S}(\mathbb{R})$ (which can be extended to $L^2$ by a standard limiting argument) with $\supp(\widehat{f})\subset \{\xi\in \mathbb{R}: 2^{m+m_0+k}\leq |\xi|\leq 2^{m+m_0+k+1}\}$ and $\supp(\widehat{g})\subset \{\eta\in \mathbb{R}: 2^m\leq |\eta|\leq 2^{m+1}\}$.

By a rescaling argument it suffices to prove that for any fixed $|k|<\mathscr{K}$ and some absolute constant $b>0$
\begin{equation}
\begin{split}
&\quad \left\| \iint\!\! \widehat{f}(\xi) \widehat{g}(\eta) e\left( \left(\xi+2^{-m_0-k}\eta\right)x \right) K(\xi, \eta)  \,\textrm{d}\xi\textrm{d}\eta\right\|_{L^1_x\left([0,2^{m+m_0+k}]\right)} \\
&\lesssim_{\mathscr{K}} 2^{bl}2^{m_0/2}2^{-m/16}\|f\|_2\|g\|_2
\end{split}\label{s4-6}
\end{equation}
for any $f, g\in \mathscr{S}(\mathbb{R})$ with $\supp(\widehat{f})\subset [1,2]$ or $[-2,-1]$ and $\supp(\widehat{g})\subset [1,2]$ or $[-2,-1]$, where
\begin{equation}
K(\xi, \eta)=\int  \!\! \tau(t) e\left(2^{m-jd_2}\left(2^{m_0+k} P_1\left(2^{j-l}t \right)\xi+ P_2\left(2^{j-l}t \right)\eta\right)\right)  \,\mathrm{d}t.   \label{s5-1}
\end{equation}

By a duality argument we conclude that to prove Proposition \ref{s2-prop1} it suffices to show that for any fixed $|k|<\mathscr{K}$ and some absolute constant $b>0$
\begin{equation}
\begin{split}
&\quad \left| \iint\!\! \widehat{f}(\xi) \widehat{g}(\eta) \widehat{h}\left(-\xi-2^{-m_0-k}\eta\right) K(\xi, \eta)  \,\textrm{d}\xi\textrm{d}\eta\right| \\
&\lesssim_{\mathscr{K}} 2^{bl} 2^{-m/2}2^{-m/16}\|f\|_2\|g\|_2\|h\|_2,
\end{split}\label{s4-4}
\end{equation}
for any $f, g, h\in \mathscr{S}(\mathbb{R})$ with $\supp(\widehat{f})\subset [1,2]$ or $[-2,-1]$ and $\supp(\widehat{g})\subset [1,2]$ or $[-2,-1]$.

Concerning the observation we have made in Remark \ref{s2-rm1}, we will additionally show that if $j-l\leq -\Gamma_1<0$ then a weaker assumption that $P_1$ and $P_2$ are linearly independent (instead of $d_1<d_2$) would suffice to yield \eqref{s4-4} with a better decay factor  $2^{-m/6}$ (instead of $2^{-m/16}$). See Case \ref{s5-subcase1} and \ref{s5-subcase2} below.


\section{The kernel $K$}\label{sec4}

In this section we deal with the special oscillatory integral \eqref{s5-1}. We will show that either its size is very small or it has an asymptotics with a leading term containing an oscillatory factor. To help achieve this goal we first change variables to transform the pair of polynomials appearing in the phase of \eqref{s5-1},
\begin{equation}
\{P_1(2^{j-l}t), P_2(2^{j-l}t)\}, 1/2\leq t\leq 2, \label{s5-2}
\end{equation}
into (up to small errors) the form of
\begin{equation*}
\{c_1\omega, c_2\omega+c_3\omega^{\nu} \}
\end{equation*}
for $\nu\neq 1$ and nonzero $c_1$ and $c_3$.

Since $|j-l|\geq \Gamma_1$, as long as $\Gamma_1$ is chosen sufficiently large, $2^{j-l}$ is sufficiently small or large at our disposal. Both polynomials in \eqref{s5-2} are dominated by their own monomials with the smallest or largest powers. Hence we need to discuss the following three cases. We emphasize that in Case \ref{s5-subcase1} and \ref{s5-subcase2} it suffices to assume that \textit{$P_1$ and $P_2$ are linearly independent} rather than $d_1<d_2$. It is in Case \ref{s5-subcase3} where we assume $d_1<d_2$.

The notations defined in this section will be used in the next two sections.

\subsection{Case $j-l\leq -\Gamma_1$ and $\sigma_1=\sigma_2$}\label{s5-subcase1}

In the first case when $j-l\ll 0$  and two polynomials have the same smallest power $\sigma:=\sigma_1=\sigma_2$, we use the substitution
\begin{equation}
P_1\left(2^{j-l}t \right)=a_{\sigma} 2^{\sigma(j-l)} \omega \label{s5-23}
\end{equation}
to transfer the pair \eqref{s5-2}. Hence $\omega=t^{\sigma}(1+O(2^{-|j-l|}))$ and
\begin{equation*}
P_2\left(2^{j-l}t \right)=b_{\sigma} 2^{\sigma(j-l)}\omega+\left(P_2\left(u\right)-\frac{b_{\sigma}}{a_{\sigma}}P_1\left(u\right)\right)\bigg|_{u=2^{j-l}t(\omega)}.
\end{equation*}
Since $P_1$ and $P_2$ are linearly independent and both terms with the power $\sigma$ are cancelled, we can write
\begin{equation*}
P_2(u)-\frac{b_{\sigma}}{a_{\sigma}}P_1(u)=c_{\varrho}u^{\varrho}+\mathcal{E}(u)
\end{equation*}
for some integer $\sigma<\varrho\leq d_2$ and nonzero constant $c_{\varrho}$, where $\mathcal{E}(u)$ is a (possibly trivial) polynomial of $u$ with all powers $\geq \varrho+1$. Therefore we can write the kernel $K$ in the following standard form
\begin{equation}
K(\xi, \eta)=\int \!\! \widetilde{\tau}(\omega) e\left(\lambda \phi(\omega, \xi, \eta) \right) \,\mathrm{d}\omega \label{s5-6}
\end{equation}
with a cut-off function $\widetilde{\tau}(\omega)=\tau(t(\omega))t'(\omega)$, a parameter
\begin{equation}
\lambda=2^{m-jd_2+\varrho(j-l)}\geq 2^{-d_2 l}2^m \label{s5-7}
\end{equation}
and a phase function
\begin{align}
\phi(\omega, \xi, \eta)&=2^{-\varrho(j-l)}\left(2^{m_0+k}P_1\left(2^{j-l}t(\omega) \right)\xi+P_2\left(2^{j-l}t(\omega) \right)\eta   \right)  \label{s5-21}\\
                       &=AC_1\omega\xi+A b_{\sigma}\omega\eta +Q(\omega)\eta,   \label{s5-8}
\end{align}
where $2^{m_0}=|b_{\sigma}|/|a_{\sigma}|$, $A=2^{(\sigma-\varrho)(j-l)}$, $C_1=\sgn\left(a_{\sigma}\right)|b_{\sigma}|2^k$ and
\begin{equation}
Q(\omega)=c_{\varrho}\omega^{\nu}+E(\omega), \quad \nu=\varrho/\sigma>1, \label{s5-16}
\end{equation}
with an error term
\begin{equation}
E(\omega)=c_{\varrho}t(\omega)^{\varrho}-c_{\varrho}\omega^{\nu}+2^{-\varrho(j-l)} \mathcal{E}\left(2^{j-l}t(\omega)\right).  \label{s5-17}
\end{equation}

It is easy to observe that
\begin{equation}
E(\omega)=2^{-\varrho(j-l)}\mathcal{P}\left(2^{j-l}t(\omega)\right) \label{s5-22}
\end{equation}
with
\begin{equation}
\mathcal{P}(u)=c_{\varrho}u^{\varrho}-c_{\varrho}a_{\sigma}^{-\nu}P_1(u)^{\nu}+\mathcal{E}(u)
=P_2(u)-\frac{b_{\sigma}}{a_{\sigma}}P_1(u)-\frac{c_{\varrho}}{a_{\sigma}^{\nu}}P_1(u)^{\nu}\label{s5-25}
\end{equation}
which is a polynomial, if $\nu\in \mathbb{N}$, of degree $\leq\max\{\nu d_1, d_2\}$. It is routine to check that $E^{(i)}(\omega)=O(2^{-|j-l|})$.


\subsection{Case $j-l\leq -\Gamma_1$ and $\sigma_1\neq \sigma_2$}\label{s5-subcase2}

In the second case when $j-l\ll 0$  and two polynomials have different smallest powers, we use the substitution
\begin{equation*}
P_1\left(2^{j-l}t \right)=a_{\sigma_1} 2^{\sigma_1(j-l)} \omega
\end{equation*}
to transfer the pair \eqref{s5-2}. Hence $\omega=t^{\sigma_1}(1+O(2^{-|j-l|}))$ and
\begin{equation*}
P_2\left(2^{j-l}t \right)=b_{\sigma_2}2^{\sigma_2(j-l)}\omega^{\nu}
+\left( P_2\left(u\right)-b_{\sigma_2}\left( \frac{P_1\left(u\right)}{a_{\sigma_1}} \right)^{\nu}\right)\bigg|_{u=2^{j-l}t(\omega)}
\end{equation*}
with $\nu=\sigma_2/\sigma_1\neq 1$. Notice that $2^{m_0}=2^{(\sigma_2-\sigma_1)(j-l)}|b_{\sigma_2}|/|a_{\sigma_1}|$. Therefore the kernel can be written as
\begin{equation}
K(\xi, \eta)=\int \!\! \widetilde{\tau}(\omega) e\left(\lambda \phi(\omega, \xi, \eta) \right) \,\mathrm{d}\omega \label{s5-3}
\end{equation}
with a cut-off function $\widetilde{\tau}(\omega)=\tau(t(\omega))t'(\omega)$, a parameter
\begin{equation}
\lambda=2^{m-jd_2+\sigma_2(j-l)}\geq 2^{-d_2 l}2^m \label{s5-4}
\end{equation}
and a phase function
\begin{equation}
\phi(\omega, \xi, \eta)=C_1\omega\xi +Q(\omega)\eta,   \label{s5-5}
\end{equation}
where $C_1=\sgn\left(a_{\sigma_1}\right)|b_{\sigma_2}|2^k$ and
\begin{equation}
Q(\omega)=b_{\sigma_2}\omega^{\nu}+ E(\omega)   \label{s5-15}
\end{equation}
with an error term
\begin{equation}
E(\omega)=2^{-\sigma_2(j-l)}\left( P_2\left(u\right)-b_{\sigma_2}\left( \frac{P_1\left(u\right)}{a_{\sigma_1}} \right)^{\nu}\right)\bigg|_{u=2^{j-l}t(\omega)} \label{s5-12}
\end{equation}
satisfying $E^{(i)}(\omega)=O(2^{-|j-l|})$.


\subsection{Case $j-l\geq \Gamma_1$ and $d_1<d_2$}\label{s5-subcase3}

In the third case when $j-l\gg 0$ and $d_1<d_2$, we use the substitution
\begin{equation*}
P_1\left(2^{j-l}t \right)=a_{d_1}2^{d_1(j-l)}\omega
\end{equation*}
to transfer the pair \eqref{s5-2}. Hence $\omega=t^{d_1}(1+O(2^{-|j-l|}))$ and
\begin{equation*}
P_2\left(2^{j-l}t \right)=b_{d_2}2^{d_2(j-l)}\omega^{\nu}
+\left( P_2\left(u\right)-b_{d_2}\left( \frac{P_1\left(u\right)}{a_{d_1}} \right)^{\nu}\right)\bigg|_{u=2^{j-l}t(\omega)}
\end{equation*}
with $\nu=d_2/d_1>1$. Notice that $2^{m_0}=2^{(d_2-d_1)(j-l)}|b_{d_2}|/|a_{d_1}|$.  Therefore the kernel can be written as
\begin{equation}
K(\xi, \eta)=\int \! \widetilde{\tau}(\omega) e\left(\lambda \phi(\omega, \xi, \eta) \right) \,\mathrm{d}\omega\label{s5-9}
\end{equation}
with a cut-off function $\widetilde{\tau}(\omega)=\tau(t(\omega))t'(\omega)$,  a parameter
\begin{equation}
\lambda=2^{-d_2 l}2^{m}\label{s5-10}
\end{equation}
and  a phase function
\begin{equation}
\phi(\omega, \xi, \eta)=C_1\omega\xi +Q(\omega)\eta,   \label{s5-11}
\end{equation}
where $C_1=\sgn\left(a_{d_1}\right)|b_{d_2}|2^k$ and
\begin{equation}
Q(\omega)=b_{d_2}\omega^{\nu}+E(\omega) \label{s5-14}
\end{equation}
with an error term
\begin{equation}
E(\omega)=2^{-d_2(j-l)}\left( P_2\left(u\right)-b_{d_2}\left( \frac{P_1\left(u\right)}{a_{d_1}} \right)^{\nu}\right)\bigg|_{u=2^{j-l}t(\omega)} \label{s5-13}
\end{equation}
satisfying $E^{(i)}(\omega)=O(2^{-|j-l|})$.

From now on we always assume $\lambda>1$ otherwise $|K(\xi,\eta)|\leq 1\leq \lambda^{-1}\leq 2^{d_2 l}2^{-m}$ and the desired \eqref{s4-4} follows immediately.

In the last part of this section we formulate results concerning the asymptotics and estimates of the kernel  $K(\xi, \eta)$. Roughly speaking, if its phase function $\phi$ (in the form of \eqref{s5-8}, \eqref{s5-5} or \eqref{s5-11}) has a (nondegenerate) critical point we apply the method of stationary phase to get an asymptotics; if not we apply integration by parts to get a rapid decay. To fulfil this idea rigorously one needs to distinguish the situations when there exists a critical point or not and be careful with the implicit constant produced by integration by parts. An example of such a discussion can be found in \cite[Lemma 3.1]{GX16} and its proof. By using the same argument we can readily get the following two lemmas.

If $\Gamma_1$ is sufficiently large then $\supp(\widetilde{\tau})\subset [2^{-\sigma-1}, 2^{\sigma+1}]$, $[2^{-\sigma_1-1}, 2^{\sigma_1+1}]$ and $[2^{-d_1-1}, 2^{d_1+1}]$ for Case \ref{s5-subcase1}, \ref{s5-subcase2} and \ref{s5-subcase3} respectively.

Let us first consider Case \ref{s5-subcase2} and \ref{s5-subcase3}. It is easy to observe that there exists a constant $C_4>2$ such that if $|\xi| \notin [C_4^{-1}, C_4]$, $|\eta|\in [1, 2]$, $\omega\in \supp(\widetilde{\tau})$ and  $\Gamma_1$ is sufficiently large then $|\partial_{\omega} \phi(\omega, \xi, \eta)|$ has a uniform lower bound. Integration by parts immediately yields
\begin{equation}
K(\xi, \eta)=O_{P_1,P_2}(\lambda^{-1}). \label{s5-24}
\end{equation}
Since we expect this bound to produce desired results easily, we will only consider $C_4^{-1}\leq |\xi|\leq C_4$ in the following lemma.

\begin{lemma}[Case \ref{s5-subcase2} and  \ref{s5-subcase3}]\label{s5-lemma1}
Let $C_4^{-1}\leq |\xi|\leq C_4$ and $1\leq |\eta|\leq 2$. Assume that $\chi\in C_c^{\infty}(\mathbb{R})$ has its support contained in an interval $I$. If $\Gamma_1$ and $|I|$ are sufficiently large and small respectively (both depending only on $P_1$ and $P_2$), then either one of the following two statements holds.

(1) We have
\begin{equation}
\chi\left(-C_1\frac{\xi}{\eta}\right)K(\xi, \eta)=O_{P_1,P_2}(\lambda^{-1}).
\label{s5-18}
\end{equation}

(2) One can choose an interval $[1/c, c]$ with $c=c(P_1)>1$ such that for each pair $(\xi, \eta)$  with $-C_1\xi/\eta\in \supp \chi$, there exists a unique point in $[1/c, c]$,
\begin{equation*}
\omega_0=\omega_0(\xi, \eta)=(Q')^{-1}\left(-C_1\frac{\xi}{\eta}\right),
\end{equation*}
such that
\begin{equation}
\partial_{\omega} \phi(\omega_0, \xi, \eta)=0   \label{s5-19}
\end{equation}
and
\begin{equation}
\begin{split}
&\chi\left(-C_1\frac{\xi}{\eta}\right)K(\xi, \eta)=\\
&\quad C\frac{\chi(-C_1\xi/\eta)\widetilde{\tau}(\omega_0)}{|\partial_{\omega\omega}^2\phi(\omega_0, \xi,
\eta)|^{1/2}}e\left(\lambda \phi(\omega_0, \xi,
\eta)\right)\lambda^{-1/2}+O_{P_1,P_2}(\lambda^{-3/2})
\end{split}\label{s5-20}
\end{equation}
with $C$ being an absolute constant.
\end{lemma}

After modifying this lemma a little, we get the analogous result for Case \ref{s5-subcase1}.

\begin{lemma}[Case \ref{s5-subcase1}]\label{s5-lemma2}
Let $1\leq |\xi|\leq 2$ and $1\leq |\eta|\leq 2$. Then the statement of Lemma \ref{s5-lemma1} is still valid  after we replace every $-C_1\frac{\xi}{\eta}$ by $-AC_1\frac{\xi}{\eta}-Ab_{\sigma}$.
\end{lemma}

\begin{remark}
The constant $c(P_1)$ can be chosen to be $2^{2\sigma+1}$, $2^{2\sigma_1+1}$ and $2^{2d_1+1}$ for Case \ref{s5-subcase1}, \ref{s5-subcase2} and \ref{s5-subcase3} respectively.
\end{remark}

In later applications one can use a partition of unity to restrict the domain of $\xi/\eta$. If \eqref{s5-18} holds the situation is easy---the fast decay in $\lambda$ will yield good bounds. If \eqref{s5-20} holds the situation is more difficult. While the error term is easy as well, one needs to take advantage of the oscillatory term in \eqref{s5-20} to obtain desired bounds.


\section{Case \ref{s5-subcase1}} \label{sec5}

To prove \eqref{s4-4} we first insert
\begin{equation*}
\sum_s \chi_s \left(-AC_1\frac{\xi}{\eta}-Ab_{\sigma} \right)
\end{equation*}
into the integrand of \eqref{s4-4}, where  $\{\chi_s \}$ is a partition of unity associated to a finite open cover of the interval $-A(C_1[1/2,2]+b_{\sigma})$ (or $-A(C_1[-2,-1/2]+b_{\sigma})$) by using open intervals of certain fixed length (depending on $P_1$ and $P_2$; see Lemma \ref{s5-lemma2}). Such manipulation does not change \eqref{s4-4} because of the size restrictions on $\xi$ and $\eta$.

For each $s$, Lemma \ref{s5-lemma2} ensures that
\begin{equation}
\chi_s\left(-AC_1\frac{\xi}{\eta}-Ab_{\sigma}\right)K(\xi,\eta) \label{s7-15}
\end{equation}
has either a rapid decay in $\lambda$ or an asymptotics. If the former situation occurs for a $s$, the contribution to the integral in \eqref{s4-4} corresponding to that $s$ is small. In fact the total contribution of those $s$'s is bounded by
\begin{align}
&\quad A\lambda^{-1}\!\! \iint\!\! \left|\widehat{f}(\xi) \widehat{g}(\eta) \widehat{h}\left(-\xi-2^{-m_0-k}\eta\right)\right| \textrm{d}\xi\textrm{d}\eta\nonumber\\
&\lesssim 2^{(d_2+\varrho-\sigma)l}2^{-m}\|f\|_2\|g\|_2\|h\|_2, \label{s7-14}
\end{align}
where we have used $2^{(\sigma-\varrho)j}\leq 1$, \eqref{s5-7} and H\"older's inequality in the last inequality.

We will next focus on the situation when \eqref{s7-15} has an asymptotics. It is easy to observe that the number of such $s$'s is $\lesssim 1$. Indeed, in view of \eqref{s5-8} and \eqref{s5-16}, the existence of a critical point of size $O(1)$ requires that $|AC_1\frac{\xi}{\eta}+Ab_{\sigma}|\lesssim 1$. Hence the observation follows.

Let us arbitrarily fix a $s$ and denote by $\omega_0$ the critical point satisfying \eqref{s5-19}. Then \eqref{s7-15} has a leading term containing $\omega_0$ and an error term $O(\lambda^{-1})$.  The contribution of the error term is also bounded by \eqref{s7-14}. Concerning the leading term we will prove for some absolute constant $\mathfrak{b}_1$
\begin{equation}
\begin{split}
&\quad \lambda^{-\frac{1}{2}}\left| \iint\!\! \widehat{f}(\xi) \widehat{g}(\eta) \widehat{h}\left(-\xi-2^{-m_0-k}\eta\right) a\left(\xi, \eta \right) e\left(\lambda \Phi(\xi, \eta)\right)  \,\textrm{d}\xi\textrm{d}\eta\right| \\
&\lesssim_{\mathscr{K}} 2^{\mathfrak{b}_1 l}2^{-\frac{1}{2}m}2^{-\frac{1}{6}m}\|f\|_2\|g\|_2\|h\|_2,
\end{split}\label{s7-16}
\end{equation}
where $\Phi(\xi, \eta)=\phi(\omega_0, \xi, \eta)$ and
\begin{equation*}
a\left(\xi, \eta \right)=\chi_s\left(-AC_1\frac{\xi}{\eta}-Ab_{\sigma}\right)\widetilde{\tau}(\omega_0)
\left|\partial^2_{\omega\omega}\phi(\omega_0, \xi,\eta)\right|^{-1/2}.
\end{equation*}
The bounds \eqref{s7-14} and \eqref{s7-16} yield \eqref{s4-4} with a better decay factor  $2^{-m/6}$ (instead of $2^{-m/16}$) for Case \ref{s5-subcase1}.

To prove \eqref{s7-16}, by changing variables $\xi+2^{-m_0-k}\eta\rightarrow \xi$ and $\eta\rightarrow \eta$ and using the duality of $L^2$, it suffices to prove that
\begin{equation}
\begin{split}
&\quad\left\|\int \!\! \widehat{f}(\xi-2^{-m_0-k}\eta) \widehat{g}(\eta) a\!\left(\xi-2^{-m_0-k}\eta, \eta \right)\!e\!\left(\lambda \Phi(\xi-2^{-m_0-k}\eta, \eta)\right)\!\textrm{d}\eta\right\|_{L^2_\xi} \\
&\lesssim_{\mathscr{K}} 2^{\mathfrak{b}_1l}\lambda^{\frac{1}{2}}2^{-\frac{1}{2}m}2^{-\frac{1}{6}m}\|f\|_2\|g\|_2.
\end{split}\label{s7-1}
\end{equation}
After proper manipulation the square of the $L^2$ norm in \eqref{s7-1} can be written as
\begin{equation}
\int \! \mathrm{d}\zeta \! \iint \!\! F_{\zeta}(\xi)G_{\zeta}(\eta) \psi_{\zeta}(\xi,\eta) e\left(
\lambda\zeta \frac{P_{\zeta}(\xi, \eta)}{\zeta}\right)   \,\mathrm{d}\xi\mathrm{d}\eta,  \label{s4-5}
\end{equation}
where
\begin{equation*}
F_{\zeta}(\xi)=\widehat{f}(\xi-2^{-m_0-k}\zeta)\overline{\widehat{f}(\xi)},
\end{equation*}
\begin{equation*}
G_{\zeta}(\eta)=\widehat{g}(\eta+\zeta)\overline{\widehat{g}(\eta)},
\end{equation*}
\begin{equation*}
\psi_{\zeta}(\xi,\eta)=a\left(\xi-2^{-m_0-k}\zeta, \eta+\zeta \right)\overline{a\left(\xi, \eta \right)}
\end{equation*}
and
\begin{equation*}
P_{\zeta}(\xi, \eta)=\Phi\left(\xi-2^{-m_0-k}\zeta, \eta+\zeta\right)-\Phi(\xi, \eta).
\end{equation*}

We would like to estimate \eqref{s4-5} by using H\"ormander's \cite[Theorem 1.1]{Hormander73} with an explicit constant (see Theorem \eqref{app1-1}). Hence we need to estimate sizes of $\partial_{\xi\eta}^2 P_{\zeta}$ and derivatives of $P_{\zeta}$ and $\psi_{\zeta}$. In fact we claim that if $\Gamma_1$ is chosen sufficiently large then there exists an integer $\mathfrak{k}=\mathfrak{k}(P_1,P_2)$  such that
\begin{equation}
\left|\partial_{\xi\eta}^2 \left(2^{-m_0-k}\partial_{\xi}-\partial_{\eta}\right)\Phi(\xi, \eta)\right|\asymp 2^{\mathfrak{k}(j-l)}, \label{s7-2}
\end{equation}
\begin{equation}
\partial_{\xi\xi\eta}^3 \left(2^{-m_0-k}\partial_{\xi}-\partial_{\eta}\right)\Phi(\xi, \eta)\lesssim 2^{(\sigma-\varrho+\mathfrak{k})(j-l)}, \label{s7-3}
\end{equation}
\begin{equation}
\partial_{\xi\xi\xi\eta}^4 \left(2^{-m_0-k}\partial_{\xi}-\partial_{\eta}\right)\Phi(\xi, \eta)\lesssim 2^{(2\sigma-2\varrho+\mathfrak{k})(j-l)} \label{s7-11}
\end{equation}
and
\begin{equation}
\frac{\partial^i \psi_{\zeta}}{\partial \xi^i}\lesssim 2^{i(\sigma-\varrho)(j-l)}, i=0,1,2. \label{s7-13}
\end{equation}
We will prove \eqref{s7-2}--\eqref{s7-13} at the end of this section. As an easy consequence we have
\begin{equation*}
\left|\partial_{\xi\eta}^2 \frac{P_{\zeta}}{\zeta}\right|\asymp 2^{\mathfrak{k}(j-l)},
\end{equation*}
\begin{equation*}
\partial_{\xi\xi\eta}^3 \frac{P_{\zeta}}{\zeta}\lesssim 2^{(\sigma-\varrho+\mathfrak{k})(j-l)}
\end{equation*}
and
\begin{equation*}
\partial_{\xi\xi\xi\eta}^4 \frac{P_{\zeta}}{\zeta}\lesssim 2^{(2\sigma-2\varrho+\mathfrak{k})(j-l)}.
\end{equation*}

By splitting the integral \eqref{s4-5} into two with respect to $\zeta$, i.e. $|\zeta|\leq \zeta_0$ and $|\zeta|>\zeta_0$, and applying trivial estimate and Theorem \eqref{app1-1} respectively, we get
\begin{equation*}
\eqref{s4-5}\lesssim_{\mathscr{K}} \left(2^{(\sigma-\varrho-\mathfrak{k})(j-l)}+1\right)\left(\zeta_0+\left(\lambda \zeta_0\right)^{-1/2} \right)\|f\|_2^2 \|g\|_2^2
\end{equation*}
with
\begin{equation*}
\zeta_0=\lambda^{-\frac{1}{3}}\leq 2^{\frac{1}{3}d_2 l}2^{-\frac{1}{3}m}.
\end{equation*}
Notice that if $\sigma-\varrho-\mathfrak{k}\geq 0$ then $2^{(\sigma-\varrho-\mathfrak{k})(j-l)}\leq 1$; if
$\sigma-\varrho-\mathfrak{k}<0$ then $2^{(\sigma-\varrho-\mathfrak{k})(j-l)}\leq 2^{-(\sigma-\varrho-\mathfrak{k})l}$. Therefore
\begin{equation*}
\eqref{s4-5}\lesssim_{\mathscr{K}} 2^{(|\sigma-\varrho-\mathfrak{k}|+\frac{1}{3}d_2)l} 2^{-\frac{1}{3}m}\|f\|_2^2 \|g\|_2^2.
\end{equation*}
By using this bound and $\lambda^{1/2}2^{-m/2}\geq 2^{-d_2 l/2}$ we finally get \eqref{s7-1} with $\mathfrak{b}_1=(4d_2+3|\sigma-\varrho-\mathfrak{k}|)/6$. We have so far finished the proof of Case \ref{s5-subcase1} except the bounds \eqref{s7-2}--\eqref{s7-13}.

\begin{proof}[Proof of \eqref{s7-2}, \eqref{s7-3}, \eqref{s7-11} and \eqref{s7-13}]
Let us first prove the claimed lower bound \eqref{s7-2}. We start with computing some derivatives:
\begin{equation*}
\partial_{\xi}\Phi(\xi,\eta)=\partial_{\xi} \phi(\omega_0,\xi,\eta)=A C_1 \omega_0,
\end{equation*}
\begin{equation*}
\partial_{\xi\eta}^2 \left(2^{-m_0-k}\partial_{\xi}-\partial_{\eta}\right)\Phi(\xi, \eta)=A C_1 \partial_{\eta}\left(2^{-m_0-k}\partial_{\xi}\omega_0-\partial_{\eta}\omega_0\right)
\end{equation*}
and, by implicitly differentiating \eqref{s5-19},
\begin{equation*}
\partial_{\xi} \omega_0=-\frac{\partial^2_{\omega\xi}\phi}{\partial^2_{\omega\omega}\phi}(\omega_0, \xi, \eta)=-\frac{AC_1}{\partial^2_{\omega\omega}\phi(\omega_0, \xi, \eta)},
\end{equation*}
\begin{equation*}
\partial_{\eta} \omega_0=-\frac{\partial^2_{\omega\eta}\phi}{\partial^2_{\omega\omega}\phi}(\omega_0, \xi, \eta)
=-\frac{A b_{\sigma}+Q'(\omega_0)}{\partial^2_{\omega\omega}\phi(\omega_0, \xi, \eta)}=\frac{\xi}{\eta}\frac{AC_1}{\partial^2_{\omega\omega}\phi(\omega_0, \xi, \eta)}.
\end{equation*}
It is useful to observe that
\begin{equation}
\partial^2_{\omega\omega}\phi(\omega_0, \xi, \eta)=\eta Q''(\omega_0)=\eta\left(c_{\varrho}\nu(\nu-1)\omega_0^{\nu-2}+E''(\omega_0) \right)\asymp 1.\label{s7-18}
\end{equation}
Collecting these formulas yields a formula of the mixed derivative of \eqref{s7-2}, namely
\begin{align}
&\partial_{\xi\eta}^2 (2^{-m_0-k}\partial_{\xi}-\partial_{\eta})\Phi(\xi, \eta)\nonumber\\
&\quad =AC_1\partial_{\eta}\left(\frac{A(b_{\sigma}-a_{\sigma})+c_{\varrho}\nu\omega_0^{\nu-1}+E'(\omega_0)}{\partial^2_{\omega\omega}\phi(\omega_0, \xi, \eta)} \right) \label{s7-7}\\
&\quad =\frac{C_1 \eta}{(\partial^2_{\omega\omega}\phi(\omega_0, \xi, \eta))^3}\cdot A\cdot \Theta,\label{s7-10}
\end{align}
where
\begin{align*}
\Theta=&A^2\left(b_{\sigma}-a_{\sigma}\right)b_{\sigma}\left(c_{\varrho}\nu(\nu-1)(\nu-2)\omega_0^{\nu-3}+E'''\left(\omega_0\right) \right)  \\
      &+A\left(a_{\sigma}-2b_{\sigma}\right)\left(c_{\varrho}^2\nu^2(\nu-1)\omega_0^{2\nu-4}+ \Delta_1 \right)-c_{\varrho}^3\nu^4(\nu-1)\omega_0^{3\nu-5}+\Delta_2
\end{align*}
with terms $\Delta_1$ and $\Delta_2$ given by
\begin{equation*}
\begin{split}
\Delta_1=&-c_{\varrho}\nu(\nu-1)(\nu-2)\omega_0^{\nu-3}E'\left(\omega_0\right)+
2c_{\varrho}\nu(\nu-1)\omega_0^{\nu-2}E''\left(\omega_0\right)\\
&-c_{\varrho}\nu\omega_0^{\nu-1}E'''\left(\omega_0\right)+E''\left(\omega_0\right)^2-E'\left(\omega_0\right)E'''\left(\omega_0\right)
\end{split}
\end{equation*}
and
\begin{equation*}
\begin{split}
\Delta_2=&-2c_{\varrho}^2\nu^2(\nu-1)\omega_0^{2\nu-4}E'\left(\omega_0\right)
-4c_{\varrho}^2\nu^2(\nu-1)\omega_0^{2\nu-3}E''\left(\omega_0\right)\\
&+c_{\varrho}^2\nu^2\omega_0^{2\nu-2}E'''\left(\omega_0\right)
+c_{\varrho}\nu(\nu-1)(\nu-2)\omega_0^{\nu-3}E'\left(\omega_0\right)^2\\
&-2c_{\varrho}\nu\omega_0^{\nu-1}E''\left(\omega_0\right)^2
-4c_{\varrho}\nu(\nu-1)\omega_0^{\nu-2}E'\left(\omega_0\right)E''\left(\omega_0\right)\\
&+2c_{\varrho}\nu\omega_0^{\nu-1}E'\left(\omega_0\right)E'''\left(\omega_0\right)
-2E'\left(\omega_0\right)E''\left(\omega_0\right)^2+E'\left(\omega_0\right)^2 E'''\left(\omega_0\right).
\end{split}
\end{equation*}
It follows from size  estimates of derivatives of the error term $E$ that $\Delta_1, \Delta_2=O(2^{-|j-l|})$. Hence it is obvious that if $\Gamma_1$ is sufficiently large then
\begin{equation}
|\Theta|\asymp A=2^{(\sigma-\varrho)(j-l)} \quad \textrm{if $a_{\sigma}=b_{\sigma}$}  \label{s7-5}
\end{equation}
and
\begin{equation}
|\Theta|\asymp A^2=2^{2(\sigma-\varrho)(j-l)} \quad \textrm{if $\nu\neq 2$ and $a_{\sigma}\neq b_{\sigma}$}. \label{s7-6}
\end{equation}

When $\nu=2$ and $a_{\sigma}\neq b_{\sigma}$, it is not easy to determine the size of $\Theta$ based on its current form as a function of $\omega_0$. To resolve this problem we will express it as a rational function (in fact, essentially a polynomial if we only care about size estimate) of $T_0:=2^{j-l}t(\omega_0)$ where $t=t(\omega)\asymp 1$ is determined by \eqref{s5-23}.

By using \eqref{s5-19} we can simplify the numerator in \eqref{s7-7} and get
\begin{align}
&\partial_{\xi\eta}^2 (2^{-m_0-k}\partial_{\xi}-\partial_{\eta})\Phi(\xi, \eta)\nonumber\\
&\quad =A^2 C_1 a_{\sigma}\partial_{\eta}\left(\frac{P_2'-P_1'}{P_1'} \left(T_0\right)\frac{1}{\partial^2_{\omega\omega}\phi(\omega_0, \xi, \eta)} \right), \label{s7-19}
\end{align}
where we have used
\begin{equation}
\frac{\xi}{\eta}=-2^{-m_0-k}\frac{P_2'}{P_1'}\left(T_0\right) \label{s7-8}
\end{equation}
which follows from \eqref{s5-19} as well by using \eqref{s5-21}.

By differentiating \eqref{s5-22} and \eqref{s5-23} we get
\begin{equation*}
E''(\omega_0)=a_{\sigma}^2\frac{\mathcal{P}''P_1'-\mathcal{P}'P_1''}{P_1'^3}\left(T_0\right).
\end{equation*}
Thus by using formulas of $Q$ and $\mathcal{P}$ (namely, \eqref{s5-16} and \eqref{s5-25})
\begin{align*}
\mathfrak{q}\left(T_0\right):&=P_1'^3\left(T_0\right)Q''(\omega_0)=2c_{\varrho}P_1'^3\left(T_0\right)+
a_{\sigma}^2\left(\mathcal{P}''P_1'-\mathcal{P}'P_1''\right)\left(T_0\right)\\
                             &=a_{\sigma}^2\left(P_1'P_2''-P_1''P_2'\right)\left(T_0\right)
\end{align*}
is a polynomial of $T_0$ of degree $\leq d_1+d_2-3$ and $|\mathfrak{q}(T_0)|\asymp |P_1'^3(T_0)|\asymp T_0^{3\sigma-3}$. Let
\begin{equation*}
\mathfrak{p}\left(T_0\right):=\left(P_1'-P_2'\right)P_1'P_2'\left(T_0\right)
\end{equation*}
be a polynomial of $T_0$ of degree $\geq d_1+d_2-2$. Then by using \eqref{s7-19}, \eqref{s7-18} and \eqref{s7-8} we get
\begin{align}
\partial_{\xi\eta}^2 &\left(2^{-m_0-k}\partial_{\xi}-\partial_{\eta}\right)\Phi(\xi, \eta)=\frac{A^2 a_{\sigma}^2}{\xi}\partial_{\eta}\left(\frac{\mathfrak{p}}{\mathfrak{q}}\left(T_0\right) \right)\nonumber\\
&=\frac{a_{\sigma}^2 2^{(1-2\sigma)(j-l)}}{\xi \mathfrak{q}(T_0)^2}t'\left(\omega_0\right)\partial_{\eta} \omega_0\cdot \left( \mathfrak{p}'\mathfrak{q}-\mathfrak{p}\mathfrak{q}' \right)\left(T_0\right)\label{s7-12}\\
&\asymp  2^{(4-6\sigma)(j-l)}\left|\left( \mathfrak{p}'\mathfrak{q}-\mathfrak{p}\mathfrak{q}' \right)\left(T_0\right)\right|.\nonumber
\end{align}
Since $\deg(\mathfrak{q})<\deg(\mathfrak{p})$, the monomial having the largest power $\deg(\mathfrak{p})+\deg(\mathfrak{q})-1$ in the polynomial $\mathfrak{p}'\mathfrak{q}-\mathfrak{p}\mathfrak{q}'$ must have a nonzero coefficient. Therefore $\mathfrak{p}'\mathfrak{q}-\mathfrak{p}\mathfrak{q}'$ is a nontrivial polynomial whose coefficients only depend on $P_1$ and $P_2$. If $\Gamma_1$ is sufficiently large then $|T_0|\asymp 2^{j-l}$ is sufficiently small and
\begin{equation*}
|\mathfrak{p}'\mathfrak{q}-\mathfrak{p}\mathfrak{q}'|\asymp 2^{\mathfrak{x}(j-l)}
\end{equation*}
for some integer $0\leq \mathfrak{x}\leq \deg(\mathfrak{p})+\deg(\mathfrak{q})-1\leq 3d_1+3d_2-6$. To conclude, when $\nu=2$, $a_{\sigma}\neq b_{\sigma}$ and $\Gamma_1$ is sufficiently large then
\begin{equation}
\left|\partial_{\xi\eta}^2 \left(2^{-m_0-k}\partial_{\xi}-\partial_{\eta}\right)\Phi(\xi, \eta)\right|\asymp 2^{(4+\mathfrak{x}-6\sigma)(j-l)}.\label{s7-9}
\end{equation}
The \eqref{s7-10}, \eqref{s7-5}, \eqref{s7-6} and \eqref{s7-9} together lead to \eqref{s7-2}.

Upper bounds \eqref{s7-3} and \eqref{s7-11} are easy to prove. We differentiate \eqref{s7-12} for the case $\nu=2$ and $a_{\sigma}\neq b_{\sigma}$, and \eqref{s7-10} for other cases. After we check that $|t^{(i)}(\omega_0)|\asymp 1$ for $i=1,2,3$ and that $\frac{\partial^i \omega_0}{\partial \xi^i}\lesssim A^i$ for $i=1,2$, it is then routine to check the sizes of the outcome of differentiation.

It is similar but easier to check \eqref{s7-13}.
\end{proof}


\section{Case \ref{s5-subcase2} and \ref{s5-subcase3}} \label{sec6}

In this section we follow the strategy used in \cite{Li13} to prove desired bilinear estimates (\eqref{s4-6} or \eqref{s4-4}) for Case \ref{s5-subcase2} and \ref{s5-subcase3}. Since \cite{Li13}, such a strategy was often used to study variants of the bilinear Hilbert transform. For instance see \cite[Proposition 4.1]{GX16} for a bilinear estimate associated to a general curve $(t, \gamma(t))$, which is closely related to what we need here. This is why, in this paper, we transform the pair of polynomials into the form of $\{C_1\omega, Q(\omega)\}$---so that we can adopt the argument (with adjustment) from the aforementioned literature. A similar treatment can be found in \cite{DGR19}.


\subsection{A first estimate} \label{sec6-1}
As a first attempt we will use the $TT^*$ method and H\"ormander's \cite[Theorem 1.1]{Hormander73} to prove for some absolute constant $b>0$
\begin{equation}
\begin{split}
&\quad \left| \iint\!\! \widehat{f}(\xi) \widehat{g}(\eta) \widehat{h}\left(-\xi-2^{-m_0-k}\eta\right) K(\xi, \eta)  \,\textrm{d}\xi\textrm{d}\eta\right| \\
&\lesssim_{\mathscr{K}}  2^{bl}\left(2^{\frac{1}{6}m_0}+1 \right)2^{-\frac{1}{2}m}2^{-\frac{1}{6}m}\|f\|_2\|g\|_2\|h\|_2,
\end{split} \label{s6-14}
\end{equation}
which implies \eqref{s4-4} with a better decay factor  $2^{-m/6}$ (instead of $2^{-m/16}$) for Case \ref{s5-subcase2} already since
\begin{equation*}
2^{\frac{1}{6}m_0}+1\lesssim 2^{\frac{|\sigma_2-\sigma_1|}{6}l},
\end{equation*}
but is not enough for Case \ref{s5-subcase3}\footnote{We will complete the proof of Case \ref{s5-subcase3} in the next subsection.} since $m_0\gg 1$ and we cannot get rid of the $j$ (in $m_0$) trivially. The proof of \eqref{s6-14} is simply a repetition of the argument in Section \ref{sec5} with an easier estimation of derivatives, hence we provide only a sketch.

By inserting a partition of unity $\sum_s \chi_s (-C_1\xi/\eta)$ with $\# s \lesssim 1$ and using Lemma  \ref{s5-lemma1}, the contribution to the left side of \eqref{s6-14} can be separated into three parts to analyze. For that coming from the leading term of \eqref{s5-20} we will prove
\begin{equation}
\begin{split}
&\quad \lambda^{-\frac{1}{2}}\left| \iint\!\! \widehat{f}(\xi) \widehat{g}(\eta) \widehat{h}\left(-\xi-2^{-m_0-k}\eta\right) a\left(\xi, \eta \right) e\left(\lambda \Phi(\xi, \eta)\right)  \,\textrm{d}\xi\textrm{d}\eta\right|\\
&\lesssim_{\mathscr{K}} 2^{\frac{2}{3}d_2 l}2^{\frac{1}{6}m_0}2^{-\frac{1}{2}m}2^{-\frac{1}{6}m}\|f\|_2\|g\|_2\|h\|_2,
\end{split}\label{s6-15}
\end{equation}
where $\Phi(\xi, \eta)=\phi(\omega_0, \xi, \eta)$, $\omega_0$ is the critical point, and
\begin{equation*}
a\left(\xi, \eta \right)=\chi_s\left(-C_1\frac{\xi}{\eta}\right)\widetilde{\tau}(\omega_0)
\left|\partial^2_{\omega\omega}\phi(\omega_0, \xi,\eta)\right|^{-1/2}.
\end{equation*}
The contributions from \eqref{s5-18} and the error term  of \eqref{s5-20} are both bounded by $O(2^{d_2 l}2^{-m}\|f\|_2\|g\|_2\|h\|_2)$. \eqref{s6-14} then follows.

By using the same manipulation (see Section \ref{sec5}), \eqref{s6-15} is reduced to the estimation of \eqref{s4-5} with $F_{\zeta}$,  $G_{\zeta}$, $\psi_{\zeta}$ and  $P_{\zeta}$ having exactly the same forms. We claim that  if $\Gamma_1$ is chosen sufficiently large then
\begin{equation*}
\left|\partial_{\xi\eta}^2 \frac{P_{\zeta}}{\zeta (2^{-m_0}+1)}\right|\asymp_{\mathscr{K}} 1,
\end{equation*}
\begin{equation*}
\partial_{\xi\xi\eta}^3 \frac{P_{\zeta}}{\zeta(2^{-m_0}+1)}\lesssim_{\mathscr{K}} 1,
\end{equation*}
\begin{equation*}
\partial_{\xi\xi\xi\eta}^4 \frac{P_{\zeta}}{\zeta(2^{-m_0}+1)}\lesssim_{\mathscr{K}} 1
\end{equation*}
and
\begin{equation*}
\frac{\partial^i \psi_{\zeta}}{\partial \xi^i}\lesssim 1, i=0,1,2.
\end{equation*}
Indeed, these bounds follow from the following facts:
\begin{equation*}
\partial_{\xi\eta}^2 \left(2^{-m_0-k}\partial_{\xi}-\partial_{\eta}\right)\Phi(\xi, \eta)=C_1 \left(2^{-m_0-k}\partial^2_{\xi\eta}\omega_0-\partial^2_{\eta\eta}\omega_0\right);
\end{equation*}
\begin{equation*}
\partial^2_{\xi\eta}\omega_0=\frac{C_1}{\eta^2 Q''(\omega_0)^3}\left(Q''(\omega_0)^2-Q'(\omega_0)Q'''(\omega_0) \right)\asymp 1;
\end{equation*}
\begin{equation*}
\partial^2_{\eta\eta}\omega_0=\frac{C_1\xi}{\eta^3 Q''(\omega_0)^3}\left(Q'(\omega_0)Q'''(\omega_0)- 2Q''(\omega_0)^2\right)\asymp 1;
\end{equation*}
$\partial^3_{\xi\xi\eta}\omega_0$, $\partial^3_{\xi\eta\eta}\omega_0$, $\partial^4_{\xi\xi\xi\eta}\omega_0$ and  $\partial^4_{\xi\xi\eta\eta}\omega_0$ are all of size $O(1)$; $2^{-m_0}$ is either $\gg 1$ or $\ll 1$.

By applying trivial estimate and \cite[Theorem 1.1]{Hormander73}, we get
\begin{align*}
\eqref{s4-5}&\lesssim_{\mathscr{K}} \left(\zeta_0+\left(\lambda\zeta_0(2^{-m_0}+1)\right)^{-\frac{1}{2}}2^{\frac{1}{2}m_0} \right)\|f\|_2^2 \|g\|_2^2\\
            &\lesssim_{\mathscr{K}} \left(\zeta_0+\left(\lambda\zeta_0\right)^{-\frac{1}{2}}2^{\frac{1}{2}m_0} \right)\|f\|_2^2 \|g\|_2^2\\
            &\lesssim_{\mathscr{K}} 2^{\frac{1}{3}d_2 l}2^{\frac{1}{3}m_0} 2^{-\frac{1}{3}m}\|f\|_2^2 \|g\|_2^2,
\end{align*}
where we have chosen
\begin{equation*}
\zeta_0=\lambda^{-\frac{1}{3}} 2^{\frac{1}{3}m_0}\leq 2^{\frac{1}{3}m_0} 2^{\frac{1}{3}d_2 l} 2^{-\frac{1}{3}m}.
\end{equation*}
By using this bound of \eqref{s4-5} and $\lambda^{1/2}2^{-m/2}\geq 2^{-d_2 l/2}$ we get \eqref{s6-15}.


\subsection{A second estimate}\label{sec6-2}

As a second attempt we will use the $TT^*$ method and the third author's $\sigma$-uniformity method (see Appendix \ref{app2}) to prove \eqref{s4-6} for Case \ref{s5-subcase3}.

We first construct a partition of unity $\{\chi_s \}$, $\# s \lesssim 1$, associated to a finite open cover of $[-C_1 C_4, C_1 C_4]$ (with $C_4$ defined right before Lemma \ref{s5-lemma1}) by using open intervals of certain fixed length (depending on $P_1$ and $P_2$; see Lemma \ref{s5-lemma1}).

Let $\textbf{I}$ be either $[1, 2]$  or $[-2, -1]$ such that $\supp(\widehat{g})\subset \textbf{I}$. We will apply Lemma \ref{s5-lemma1} to $\chi_s(-C_1\xi/\eta)K(\xi, \eta)$ later for $C_4^{-1} \leq |\xi|\leq C_4$ and $\eta\in \textbf{I}$. Denote by $\mathfrak{S}$ the collection of all $s$'s for which the second statement in Lemma \ref{s5-lemma1} holds, and
\begin{equation*}
\mathfrak{U}(\textbf{I}):=\{u_{s, r, \xi}(\eta)\in L^2(\textbf{I}) : s\in \mathfrak{S},
r\in\mathbb{R}, C_4^{-1} \leq |\xi|\leq C_4\},
\end{equation*}
where
\begin{equation*}
u_{s, r, \xi}(\eta)=\chi_s(-C_1\xi/\eta)e\left(-\lambda \phi(\omega_0(\xi, \eta), \xi, \eta)-r\eta\right).
\end{equation*}

\textit{As a first step,} we let $\widehat{g}|_{\textbf{I}}$, the restriction of $\widehat{g}$ to $\textbf{I}$, be an arbitrary function in $L^2(\textbf{I})$ that is $\sigma$-uniform in $\mathfrak{U}(\textbf{I})$.

For simplicity we let
\begin{equation*}
B(f,g)(x)=\iint\!\! \widehat{f}(\xi) \widehat{g}(\eta) e\left( \left(\xi+2^{-m_0-k}\eta\right)x \right) K(\xi, \eta)  \,\textrm{d}\xi\textrm{d}\eta
\end{equation*}
denote the integral in \eqref{s4-6}, which in time space is equal to
\begin{equation}
B(f,g)(x)=\int\!\! f\left(x+\lambda C_1\omega\right) g\left(2^{-m_0-k}x+\lambda Q(\omega)\right)   \widetilde{\tau}(\omega)\,\textrm{d}\omega \label{s6-8}
\end{equation}
by \eqref{s5-9}. Localizing in $x$,
\begin{equation}
\int \!\! B(f,g)(x)h(x) \,\mathrm{d}x \label{s6-7}
\end{equation}
has the following decomposition
\begin{align*}
\eqref{s6-7}=\sum_{q\in\mathbb{Z}} &\iint \!\!
\left(\textbf{1}_{I_{q}}f\right)\left(x+\lambda C_1\omega\right)g\left(2^{-m_0-k}x+\lambda Q(\omega)\right)\cdot\\
&\widetilde{\tau}(\omega)\left(\textbf{1}_{2^{m_0+k}[q,q+1)}h\right)(x) \,\textrm{d}\omega\mathrm{d}x,
\end{align*}
where $I_{q}=[2^{m_0+k}q-\lambda C, 2^{m_0+k}(q+1)+\lambda C]$ for some proper constant $C>0$. In frequency space \eqref{s6-7} is equal to
\begin{equation*}
\sum_{q\in\mathbb{Z}}\! \iiint \!\!
\widehat{\textbf{1}_{I_{q}}f}(\xi)e\left(\xi x \right)\widehat{g}(\eta)e\!\left( 2^{-m_0-k}\eta x \right)\!K(\xi,\eta)\left(\textbf{1}_{2^{m_0+k}[q,q+1)}h\right)\!(x)\mathrm{d}x\textrm{d}\xi\textrm{d}\eta.
\end{equation*}
By using the power series of $e(2^{-m_0-k}\eta (x-2^{m_0+k}q))$, we then have
\begin{equation*}
\begin{split}
\eqref{s6-7}=\sum_{q\in\mathbb{Z}}\sum_{p=0}^{\infty} &\frac{(2\pi
i)^p}{p!}\iint \!\! \widehat{\textbf{1}_{I_{q}}f}(\xi) \widehat{g}(\eta)\eta^p e\left(q\eta \right) K(\xi,\eta)\cdot\\
&\mathcal{F}^{-1}\left[\left(2^{-m_0-k}\cdot-q\right)^p\left(\textbf{1}_{2^{m_0+k}[q,q+1)}h\right)(\cdot)\right](\xi)
\,\mathrm{d}\xi\mathrm{d}\eta.
\end{split}
\end{equation*}

If $|\xi| \notin [C_4^{-1}, C_4]$, \eqref{s5-24} gives $K(\xi,\eta)=O(\lambda^{-1})$. By H\"older's inequality, this portion of \eqref{s6-7} is
\begin{align}
&\lesssim \lambda^{-1}\|\widehat{g}\|_2 \sum_{q\in\mathbb{Z}}\left\|\textbf{1}_{I_{q}}f\right\|_2 \left\|\textbf{1}_{2^{m_0+k}[q,q+1)}h \right\|_2\nonumber\\
&\lesssim_{\mathscr{K}} \left(1+2^{-\frac{1}{2}m_0}\lambda^{\frac{1}{2}} \right)\lambda^{-1}\|f\|_2\|\widehat{g}\|_2\|h\|_2. \label{s6-17}
\end{align}

Hence we may next restrict the domain of $|\xi|$ in \eqref{s6-7} to $[C_4^{-1}, C_4]$ by properly adding a bump function $\widehat{\varphi_1}$. Furthermore by adding the partition of unity constructed earlier, we need to estimate, for each fixed $s$,
\begin{equation}
\begin{split}
\sum_{q\in\mathbb{Z}}\sum_{p=0}^{\infty}&\frac{(2\pi
i)^p}{p!}\iint \!\! \widehat{\textbf{1}_{I_{q}}f}(\xi)\widehat{\varphi_1}(\xi)\widehat{g}(\eta)\eta^p e\left(q\eta \right)\chi_s\left(-C_1\frac{\xi}{\eta}\right) K(\xi,\eta)\cdot\\
&\mathcal{F}^{-1}\left[\left(2^{-m_0-k}\cdot-q\right)^p\left(\textbf{1}_{2^{m_0+k}[q,q+1)}h\right)(\cdot)\right](\xi)
\,\mathrm{d}\xi\mathrm{d}\eta.
\end{split}\label{s6-4}
\end{equation}

If the first statement of Lemma \ref{s5-lemma1} holds then \eqref{s6-4} is also bounded by \eqref{s6-17}.

If the second statement holds then \eqref{s6-4} is reduced to
\begin{equation}
\begin{split}
\lambda^{-1/2}\sum_{q\in\mathbb{Z}}\sum_{p=0}^{\infty}&\frac{(2\pi
i)^p}{p!}\int \! \mathfrak{M}(\xi) \widehat{\textbf{1}_{I_{q}}f}(\xi)\cdot\\
&\mathcal{F}^{-1}\left[\left(2^{-m_0-k}\cdot-q\right)^p\left(\textbf{1}_{2^{m_0+k}[q,q+1)}h\right)(\cdot)\right](\xi)
\,\mathrm{d}\xi,
\end{split}\label{s6-5}
\end{equation}
where we have omitted the error term of \eqref{s5-20} (since it leads to the bound \eqref{s6-17} as well), and
\begin{equation*}
\mathfrak{M}(\xi)=\int_{\textbf{I}} \! b(\xi,
\eta)\widehat{g}(\eta)\chi_s(-C_1\xi/\eta) e\left(\lambda \phi(\omega_0(\xi, \eta), \xi, \eta)+q\eta \right) \,\mathrm{d}\eta
\end{equation*}
with
\begin{equation*}
b(\xi,\eta)=\widehat{\varphi_1}(\xi)\eta^p \widetilde{\tau}(\omega_0(\xi, \eta))|\eta Q''(\omega_0(\xi, \eta))|^{-1/2}.
\end{equation*}
Applying the Fourier series of $b(\xi, \eta)$ and the assumption that
$\widehat{g}|_{\textbf{I}}$ is $\sigma$-uniform in $\mathfrak{U}(\textbf{I})$ yields
\begin{equation*}
|\mathfrak{M}(\xi)|\lesssim 9^{p}\sigma\|\widehat{g}\|_2.
\end{equation*}
Hence
\begin{equation*}
|\eqref{s6-5}|\lesssim_{\mathscr{K}} \sigma\left(1+2^{-\frac{1}{2}m_0}\lambda^{\frac{1}{2}} \right)\lambda^{-\frac{1}{2}} \|f\|_2 \|\widehat{g}\|_2\|h\|_2.
\end{equation*}

Based on the above analysis (especially \eqref{s6-17} and the bound of \eqref{s6-5}), we can draw a conclusion that if $\widehat{g}|_{\textbf{I}}$ is $\sigma$-uniform in $\mathfrak{U}(\textbf{I})$ and $\sigma>\lambda^{-1/2}$ then
\begin{equation}
\begin{split}
&\quad \left|\int \!\! B(f,g)(x)\textbf{1}_{[0,2^{m+m_0+k}]}(x)h(x) \,\mathrm{d}x\right|  \\
&\lesssim_{\mathscr{K}} \sigma\left(1+2^{-\frac{1}{2}m_0}\lambda^{\frac{1}{2}} \right) 2^{\frac{m+m_0}{2}}\lambda^{-\frac{1}{2}} \|f\|_2\|\widehat{g}\|_2\|h\|_{\infty}.
\end{split}\label{s6-6}
\end{equation}

\textit{As a second step,} we now  assume that $\widehat{g}|_{\textbf{I}}\in \mathfrak{U}(\textbf{I})$. By using \eqref{s6-8}, changing variables $x\rightarrow \lambda 2^{m_0+k}x-\lambda C_1\omega$ and H\"older's inequality, we have
\begin{equation*}
|\eqref{s6-7}|\leq \lambda^{\frac{1}{2}}2^{\frac{m_0+k}{2}}\|f\|_2\|T(h)\|_2,
\end{equation*}
where
\begin{equation*}
T(h)(x)=\!\!\int  \!\! g\left(\lambda \left(x-2^{-m_0-k}C_1\omega+Q(\omega)\right)\right)
h\left(\lambda 2^{m_0+k}x-\lambda C_1\omega\right)\widetilde{\tau}(\omega)\textrm{d}\omega.
\end{equation*}
Let $\widehat{g}|_{\textbf{I}}(\eta)=u_{s, r, \xi}(\eta)$ for arbitrarily fixed $s\in \mathfrak{S}$,
$r\in\mathbb{R}$ and  $C_4^{-1} \leq |\xi|\leq C_4$. By the Fourier inversion and changing variables we get
\begin{equation}
\|T(h)\|_2^2=\!\int \! \left|\!\int  \!\!  K_{1}(x, \omega)  h\left(\lambda 2^{m_0+k}x+r2^{m_0+k}-\lambda C_1\omega\right) \widetilde{\tau}(\omega)\,\textrm{d}\omega\right|^2 \! \mathrm{d}x, \label{s6-19}
\end{equation}
where
\begin{equation}
K_{1}(x, \omega)=\int_{\textbf{I}} \! \chi_s\left(-C_1\frac{\xi}{\eta}\right)e\left(-
\lambda \left(\phi(\omega_0(\xi, \eta), \xi, \eta)-y(x, \omega)\eta \right)\right) \, \textrm{d} \eta \label{s6-9}
\end{equation}
with
\begin{equation*}
y(x, \omega)=x-2^{-m_0-k}C_1\omega+Q(\omega).
\end{equation*}

We split \eqref{s6-19} into two parts with respect to $x$. When $|x|$ is sufficiently large, integration by parts gives $K_{1}(x, \omega)=O(\lambda^{-1}|x|^{-1})$. Hence the part for large $|x|$ is of size
\begin{equation}
O\left(\lambda^{-2}\|h\|_{\infty}^2 \right). \label{s6-20}
\end{equation}

As to the part for $|x|<M$ for some fixed large constant $M$, since the second derivative with respect to $\eta$ of the phase function of $K_{1}(x, \omega)$ is $\asymp 1$, we can argue similarly as in Lemma \ref{s5-lemma1} and Section \ref{sec5} and assume without loss of generality that there exists only one critical point $\eta_0=\eta_0(x, \omega,\xi)\in \textbf{I}$ whose defining equation is
\begin{equation*}
Q\left(\omega_0(\xi, \eta_0)\right)-y(x, \omega)=0,
\end{equation*}
otherwise integration by parts produces the bound \eqref{s6-20} again. Recall that $\omega_0(\xi,\eta)$ satisfies $C_1\xi+Q'(\omega_0(\xi,\eta))\eta=0$. Hence
\begin{equation*}
\eta_0=-\frac{C_1\xi}{Q'(Q^{-1}(y(x, \omega)))}.
\end{equation*}
By using the method of stationary phase and integration by parts we get
\begin{equation*}
\begin{split}
K_{1}(x, \omega)=&C\chi_s\left(-\frac{C_1\xi}{\eta_0}\right) \left| C_1\xi \cdot \frac{Q''}{Q'^3}\left(Q^{-1}(y(x, \omega))\right) \right|^{1/2}\cdot \\
   &e\left(-\lambda C_1\xi Q^{-1}\left(y(x, \omega)\right) \right)\lambda^{-1/2}+O\left(\lambda^{-1}\right)
\end{split}
\end{equation*}
with an absolute constant $C$. Hence $\|T(h)\|_2^2$ is reduced to
\begin{equation}
\begin{split}
\lambda^{-1}\int \! \chi_M(x)\bigg| \! \int\! &h\left(\lambda 2^{m_0+k}x+r2^{m_0+k}-\lambda C_1\omega\right)\cdot \\
&k(x, \omega)e\left(-\lambda C_1\xi Q^{-1}\left(y(x, \omega)\right) \right)
\,\mathrm{d}\omega\bigg|^2 \,\mathrm{d}x,
\end{split}\label{s6-10}
\end{equation}
where $\chi_M$ is a standard bump function supported in $[-M, M]$ and
\begin{equation*}
k(x, \omega)= \chi_s\left(-\frac{C_1\xi}{\eta_0}\right) \left| \frac{Q''}{Q'^3}\left(Q^{-1}(y(x, \omega))\right) \right|^{1/2}\widetilde{\tau}(\omega).
\end{equation*}
By using the $TT^*$ method and changing variables $\omega_1=\upsilon+\zeta$, $\omega_2=\upsilon$ and
$x\rightarrow x+2^{-m_0-k}C_1\upsilon$, we have
\begin{equation*}
\eqref{s6-10}=\lambda^{-1}\int \textrm{d}\zeta \int \!\! H_{\zeta}(x) \,\mathrm{d}x \int \!\!
K_{\zeta, x}(\upsilon)e\left(-\lambda C_1\xi P_{\zeta, x}(\upsilon) \right)
\,\mathrm{d}\upsilon,
\end{equation*}
where all three integrals are over some finite intervals,
\begin{equation*}
H_{\zeta}(x)=h\left(\lambda 2^{m_0+k}x+r2^{m_0+k}-\lambda C_1\zeta\right)\overline{h\left(\lambda 2^{m_0+k}x+r2^{m_0+k}\right)},
\end{equation*}
\begin{equation*}
K_{\zeta,
x}(\upsilon)\!=\!\chi_M\!\left(x+2^{-m_0-k}C_1\upsilon\right)\!k(x+2^{-m_0-k}C_1\upsilon,
\upsilon+\zeta)\overline{k(x+2^{-m_0-k}C_1\upsilon, \upsilon)}          
\end{equation*}
and
\begin{equation*}
P_{\zeta, x}(\upsilon)=\Phi(x-2^{-m_0-k}C_1\zeta, \upsilon+\zeta)-\Phi(x, \upsilon)
\end{equation*}
with
\begin{equation*}
\Phi(x,\upsilon)=Q^{-1}(x+Q(\upsilon)).
\end{equation*}

We have that if $2^{-m_0-k}/|x|$ is sufficiently small then
\begin{equation}
\left|D_{\upsilon}\frac{P_{\zeta, x}}{|x||\zeta|}\right|\asymp 1\label{s6-12}
\end{equation}
and
\begin{equation}
D^2_{\upsilon\upsilon}\frac{P_{\zeta, x}}{|x||\zeta|}\lesssim 1.
\label{s6-13}
\end{equation}
Indeed, the \eqref{s6-12} follows from
\begin{equation*}
\frac{\partial^2 \Phi}{\partial x\partial \upsilon}(x,\upsilon)=-Q'(\upsilon)\frac{Q''}{Q'^3}\left(Q^{-1}(x+Q(\upsilon))\right)
\asymp 1,
\end{equation*}
\begin{align*}
\frac{\partial^2 \Phi}{\partial \upsilon^2}(x,
\upsilon)&=\frac{Q''(\upsilon)}{Q'(Q^{-1}(\theta'))}\cdot \frac{Q''}{Q'^3}\left(Q^{-1}(x+Q(\upsilon))\right)\cdot
\frac{Q'(2Q''^2-Q'Q''')}{Q''^2}(\theta)\cdot x\\
&\asymp |x|,
\end{align*}
where $\theta$ and $Q^{-1}(\theta')$ are both between $Q^{-1}(x+Q(\upsilon))$ and $\upsilon$, and the fact
\begin{equation*}
x+Q(\upsilon)=y\left(x+2^{-m_0-k}C_1\upsilon, \upsilon\right)=Q\left(\omega_0\left(\xi, \eta_0\left(x+2^{-m_0-k}C_1\upsilon, \upsilon,\xi\right)\right)\right).
\end{equation*}
The \eqref{s6-13} can be proved similarly.

Hence if $2^{-m_0-k}/|x|$ is sufficiently small, for any $\mathfrak{r}<1$ we have
\begin{equation*}
\int \!\! K_{\zeta, x}(\upsilon)e\left(-\lambda C_1\xi P_{\zeta, x}(\upsilon) \right)
\,\mathrm{d}\upsilon\lesssim \min\{1, (\lambda |x||\zeta|)^{-1}\}\leq (\lambda |x||\zeta|)^{-\mathfrak{r}}.
\end{equation*}
By splitting \eqref{s6-10} into two parts with respect to $x$ and applying trivial estimate and the above bound respectively, we get
\begin{equation*}
|\eqref{s6-10}|\lesssim_{\mathscr{K}, \mathfrak{r}} \lambda^{-1}\left( 2^{-m_0}+\lambda^{-\mathfrak{r}} \right)\|h\|_{\infty}^2.
\end{equation*}

To conclude the second step, we have shown that if $\widehat{g}|_{\textbf{I}}\in \mathfrak{U}(\textbf{I})$ then
\begin{equation}
\begin{split}
&\quad \left|\int \!\! B(f,g)(x)\textbf{1}_{[0,2^{m+m_0+k}]}(x)h(x) \,\mathrm{d}x\right|  \\
&\lesssim_{\mathscr{K}, \mathfrak{r}} 2^{\frac{1}{2}m_0} \left(2^{-\frac{1}{2}m_0}+\lambda^{-\frac{1}{2}\mathfrak{r}}\right) \|f\|_2\|h\|_{\infty}.
\end{split} \label{s6-18}
\end{equation}

\textit{As a final step,} we take advantage of the bounds \eqref{s6-6} and \eqref{s6-18} from the above two steps, discuss in several cases and apply Lemma \ref{app2-1} (with $\sigma$ properly chosen) to conclude Case \ref{s5-subcase3}.

If $2^{-m_0}\leq \lambda^{-1}$ then
\begin{equation*}
\left\| B(f,g)\right\|_{L^1_x\left([0,2^{m+m_0+k}]\right)}\lesssim_{\mathscr{K}} 2^{\frac{1}{2}d_2 l} 2^{\frac{1}{2}m_0}2^{-\frac{\mathfrak{r}}{4} m}\|f\|_2\|g\|_2,
\end{equation*}
which obviously ensures \eqref{s4-6}.

If $\lambda^{-1}<2^{-m_0}<\lambda^{-\mathfrak{r}}$ then
\begin{equation*}
\left\| B(f,g)\right\|_{L^1_x\left([0,2^{m+m_0+k}]\right)}\lesssim_{\mathscr{K}} 2^{\frac{1}{2}d_2 l} 2^{\frac{1}{2}m_0}2^{\frac{1-2\mathfrak{r}}{4} m}\|f\|_2\|g\|_2,
\end{equation*}
which obviously ensures \eqref{s4-6} as well.

If $2^{-m_0}\geq \lambda^{-\mathfrak{r}}$ then
\begin{equation*}
\left\| B(f,g)\right\|_{L^1_x\left([0,2^{m+m_0+k}]\right)}\lesssim_{\mathscr{K}} 2^{\frac{1}{2}m_0}\left(2^{-\frac{1}{2}m_0}2^{\frac{1}{4} m}\right)\|f\|_2\|g\|_2.
\end{equation*}
Note also that \eqref{s6-14} implies
\begin{equation*}
\left\| B(f,g)\right\|_{L^1_x\left([0,2^{m+m_0+k}]\right)}\lesssim_{\mathscr{K}} 2^{\frac{1}{2}m_0}\left(2^{\frac{1}{6}m_0}2^{-\frac{1}{6} m}\right)\|f\|_2\|g\|_2.
\end{equation*}
Balancing these two bounds yields the factor $2^{m_0/2}2^{-m/16}$ in \eqref{s4-6} for Case \ref{s5-subcase3}, hence finishes the proof of Proposition \ref{s2-prop1}.


\section{Proof of Theorem \ref{thm2}}\label{sec7}

To prove Theorem \ref{thm2}, we first observe that it suffices to prove that, given two linearly independent polynomials $P_1$ and $P_2$ such that $P_1(0)=P_2(0)=0$ and $P_1$ is strictly increasing on an interval $(0, c)$ for some small constant $c>0$, the given set $E$ contains a triplet
\begin{equation*}
x, x+t, x+\gamma(t)
\end{equation*}
for some $0<t<P_1(c)$, where we define
\begin{equation*}
\gamma(t)=P_2\circ P_1^{-1}(t) \chi(t),
\end{equation*}
where $\chi(t)=0$ if $t<0$; $=1$ if $0\leq t\leq P_1(c)$; $=0$ if $t\geq 2P_1(c)$; is smooth away from the origin. In particular, $\gamma$ is continuous on $\mathbb{R}$ and exactly the composition of $P_2$ and the inverse of $P_1$ on $[0, P_1(c)]$. The strictly decreasing $P_1$ case follows easily.

Formally this problem is the same as those considered in \cite{LP09, FGP19}. It is not hard to check that it can be proved by following the argument in \cite{LP09, FGP19} without too many changes provided one can generalize the Sobolev improving estimate in \cite{FGP19} from a polynomial $P$ to the above $\gamma$. Hence we will focus on this generalization below.

Let $\vartheta$ be a nonnegative smooth bump function supported in $[1,2]$ with integral $1$. Let $\vartheta_l(t)=2^l\vartheta(2^l t)$. For $l\in\mathbb{N}$ with $2^{1-l}<P_1(c)$ and $f$, $g$ in the Schwartz space $\mathscr{S}(\mathbb{R})$, we set
\begin{equation*}
T_l(f,g)(x):=\int_{\mathbb{R}} f(x+t)g(x+\gamma(t))\vartheta_l(t)\,\mathrm{d}t, \quad x\in \mathbb{R}.
\end{equation*}
With the above definitions we will prove the following Sobolev improving estimate involving the Sobolev norm
\begin{equation*}
\|f\|_{H^s}:=\left(\int_{\mathbb{R}}\left|\widehat{f}(\xi)\right|^2\left(1+|\xi|^2\right)^s\mathrm{d}\xi\right)^{1/2}.
\end{equation*}

\begin{proposition}\label{s8-prop1}
There exist constants $L$ and $\mathfrak{b}$ depending only on $P_1$ and $P_2$ such that for any $0<s<1/18$ we have
\begin{equation*}
\|T_{\sigma_1 l}(f,g)\|_{H^{s}}\lesssim_{s} 2^{\mathfrak{b} l}\|f\|_{H^{-s}}\|g\|_{H^{-s}}
\end{equation*}
whenever $l\geq L$ and $f, g\in \mathscr{S}(\mathbb{R})$.
\end{proposition}

\begin{remark}
Recall that $P_1(t)$ and $P_2(t)$ are denoted by \eqref{s2-8} and \eqref{s2-9}. The number $\sigma_1$ is the smallest power in $P_1$. Actually this proposition still holds if $\sigma_1$ is removed from its statement. However its current form is good enough for our purpose and its proof is a direct application of our work in previous sections.

Instead of assuming $d_1<d_2$ (for Theorem \ref{thm1}), we only assume $P_1$ and $P_2$ are linearly independent (for Theorem \ref{thm2}). This difference is due to that we consider small $t$ and use results from Case \ref{s5-subcase1} and \ref{s5-subcase2} only.
\end{remark}

\begin{proof}[Proof of Proposition \ref{s8-prop1}]
Define projections $\{\mathfrak{P}_{k}^{cl}\}_{k\in \mathbb{Z}_+}$  by
\begin{equation*}
\widehat{\mathfrak{P}_{k}^{cl}f}(\xi)=\widehat{f}(\xi)\textbf{1}_{(2^{k-1},2^k]}\left(|\xi|2^{-cl}\right), \quad \textrm{if $k\geq 1$},
\end{equation*}
and
\begin{equation*}
\widehat{\mathfrak{P}_{0}^{cl}f}(\xi)=\widehat{f}(\xi)\textbf{1}_{(0, 1]}\left(|\xi|2^{-cl}\right).
\end{equation*}
We will omit the superscript $cl$ in $\mathfrak{P}_{k}^{cl}$ when $c=0$.

For $h\in \mathscr{S}(\mathbb{R})$ we consider
\begin{equation}
\iint\! f(x+\omega)g(x+\gamma(\omega))h(x)\vartheta_{\sigma_1 l}(\omega)\,\mathrm{d}\omega\mathrm{d}x.    \label{s8-2}
\end{equation}
Using the above projections and the Fourier inversion, we have
\begin{equation*}
\eqref{s8-2}=\sum_{k_1,k_2,k_3\in\mathbb{Z}_+}\!\!\! \iint\! \widehat{\mathfrak{P}_{k_1}^{\sigma_1 l}f}(\xi)\widehat{\mathfrak{P}_{k_2}^{\sigma_2 l}g}(\eta)\widehat{\mathfrak{P}_{k_3}h}(-\xi-\eta)\mathfrak{m}_l(\xi,\eta)\,\mathrm{d}\xi\mathrm{d}\eta,
\end{equation*}
where\footnote{The counterpart of this $\mathfrak{m}_l(\xi,\eta)$ in previous sections is \eqref{s4-2} with $j=0$.}
\begin{equation*}
\mathfrak{m}_l(\xi,\eta)=\int
\! \vartheta(\omega) e\left(\xi 2^{-\sigma_1 l}\omega+\eta\gamma\left(2^{-\sigma_1 l}\omega\right)\right)\,\mathrm{d}\omega.
\end{equation*}
We may further assume that $0\leq k_3\leq \max\{k_1,k_2\}+(\sigma_1+\sigma_2)l+1$ otherwise $\widehat{\mathfrak{P}_{k_3}h}(-\xi-\eta)$ vanishes.

We will split \eqref{s8-2} into several parts (depending on sizes of $k_1$, $k_2$) and then estimate them one by one. We use $\mathfrak{b}$ to represent a constant depending on $P_1$ and $P_2$, which may be different from line to line.

First, we consider the part with $k_1\geq \mathscr{K}$ and $0\leq k_2\leq k_1-\mathscr{K}$ where $\mathscr{K}\in\mathbb{N}$ is a sufficiently large constant such that
\begin{equation*}
\partial_{\omega}\left(\xi 2^{-\sigma_1 l}\omega+\eta\gamma\left(2^{-\sigma_1 l}\omega\right)\right) \asymp 2^{k_1}
\end{equation*}
when $|\xi|2^{-\sigma_1 l}\asymp 2^{k_1}$ and $|\eta|2^{-\sigma_2 l}\leq 2^{k_2}$. An integration by parts then yields
\begin{equation*}
\mathfrak{m}_l(\xi,\eta)\lesssim 2^{-k_1}.
\end{equation*}
Using this estimate and H\"older's inequality, we have the first part bounded by
\begin{align*}
&\sum_{k_1\geq \mathscr{K}}\!\!\! 2^{-\frac{1}{2}k_1+\mathfrak{b}l}\!\!\left(\!\int_{|\xi|\leq 2^{k_1+\sigma_1 l}}\!\! |\widehat{f} |^2  \!\!\right)^{\!\frac{1}{2}}\!\left(\!\int_{|\eta|\leq 2^{k_1+\sigma_2 l-\mathscr{K}}}\! \!|\widehat{g} |^2   \!\!\right)^{\!\frac{1}{2}}\!\left(\!\int_{|\xi|\leq 2^{k_1+(\sigma_1+\sigma_2) l+1}}\! \!|\widehat{h}|^2 \!\!\right)^{\!\frac{1}{2}}\\
&\lesssim 2^{\mathfrak{b}l}\sum_{k_1\geq \mathscr{K}} 2^{-(\frac{1}{2}-3s)k_1}\|f\|_{H^{-s}}\|g\|_{H^{-s}}\|h\|_{H^{-s}}\\
&\lesssim 2^{\mathfrak{b}l}\|f\|_{H^{-s}}\|g\|_{H^{-s}}\|h\|_{H^{-s}}.
\end{align*}
It is easy to check that the part with $k_2\geq \mathscr{K}$ and $0\leq k_1\leq k_2-\mathscr{K}$ can be handled similarly.

Second, by H\"older's inequality the part with $k_1=0$ and $0\leq k_2<\mathscr{K}$ or $k_2=0$ and $0\leq k_1<\mathscr{K}$ is bounded by
\begin{align*}
&\quad 2^{\mathfrak{b}l}\left(\!\int_{|\xi|\leq 2^{\mathscr{K}+\sigma_1 l}}\! |\widehat{f} |^2  \!\right)^{\!\frac{1}{2}}\!\left(\!\int_{|\eta|\leq 2^{\mathscr{K}+\sigma_2 l}}\! |\widehat{g} |^2   \!\right)^{\!\frac{1}{2}}\!\left(\!\int_{|\xi|\leq 2^{\mathscr{K}+(\sigma_1+\sigma_2) l+1}}\! |\widehat{h}|^2 \!\right)^{\!\frac{1}{2}}\\
&\lesssim 2^{\mathfrak{b}l} \|f\|_{H^{-s}}\|g\|_{H^{-s}}\|h\|_{H^{-s}}.
\end{align*}

Third, it remains to estimate the part with $k_1,k_2\geq 1$ and $|k_1-k_2|<\mathscr{K}$, namely
\begin{equation}
\sum_{|k|<\mathscr{K}}\sum_{\substack{k_1=k_2+k\geq 1 \\k_2\geq 1, k_3\in \mathbb{Z}_+}} \iint\! \widehat{\mathfrak{P}_{k_1}^{\sigma_1 l}f}(\xi)\widehat{\mathfrak{P}_{k_2}^{\sigma_2 l}g}(\eta)\widehat{\mathfrak{P}_{k_3}h}(-\xi-\eta)\mathfrak{m}_l(\xi,\eta)\,\mathrm{d}\xi\mathrm{d}\eta.\label{s8-3}
\end{equation}
Let us denote
\begin{equation*}
F=\mathfrak{P}_{k_1}^{\sigma_1 l}f, \, G=\mathfrak{P}_{k_2}^{\sigma_2 l}g, \, H=\mathfrak{P}_{k_3}h,
\end{equation*}
\begin{equation*}
c_1=2^{k_1+\sigma_1 l-1}, \, c_2=2^{k_2+\sigma_2 l-1}
\end{equation*}
and set $\phi_c(x)=c^{-1}\phi(c^{-1}x)$. After changing variables the integral in \eqref{s8-3} is equal to
\begin{equation}
c_1 c_2 \iint\! \widehat{F_{c_1}}(\xi)\widehat{G_{c_2}}(\eta)\widehat{H_{c_1}}(-\xi-2^{-m_0-k}\eta)
K(\xi,\eta)\,\mathrm{d}\xi\mathrm{d}\eta,\label{s8-4}
\end{equation}
where $m_0=(\sigma_1-\sigma_2)l$ and
\begin{equation}
K(\xi,\eta)=\int \! \vartheta(\omega) e\left(2^{k_2+\sigma_2 l-1}\left( 2^{m_0+k}2^{-\sigma_1 l}\omega\xi+\gamma\left(2^{-\sigma_1 l}\omega\right)\eta\right)\right)\,\mathrm{d}\omega. \label{s8-5}
\end{equation}
Note that $\supp(\widehat{F_{c_1}})$ and $\supp(\widehat{G_{c_2}})$ are now both contained in $\{\xi\in\mathbb{R} :1\leq |\xi|\leq 2 \}$.  Via a substitution $P_1(2^{-l}t)=2^{-\sigma_1 l}\omega$ we note also that the kernel \eqref{s8-5} is essentially just \eqref{s5-1} with $j=0$ and $m=k_2+\sigma_2 l-1$.  Hence if $l\geq L$ for a sufficiently large $L=L(P_1, P_2)$ then $j-l=-l\leq -L$ and \eqref{s4-4} (together with Remark \ref{s2-rm1}) implies
\begin{equation*}
|\eqref{s8-4}|\lesssim_{\mathscr{K}} 2^{\mathfrak{b}l} 2^{-\frac{1}{6}k_2}\|\mathfrak{P}_{k_1}^{\sigma_1 l}f\|_2\|\mathfrak{P}_{k_2}^{\sigma_2 l}g\|_2\|\mathfrak{P}_{k_3}h\|_2.
\end{equation*}
(Instead of transforming \eqref{s8-5} into the form of \eqref{s5-1} and then applying \eqref{s4-4} directly, one can also obtain the above bound by following the steps in Subsection \ref{s5-subcase1} and \ref{s5-subcase2} to write the kernel in the form of \eqref{s5-6} and \eqref{s5-3} and then applying results from Section \ref{sec5} and Subsection \ref{sec6-1}.) Using this bound of \eqref{s8-4} we readily get
\begin{align*}
|\eqref{s8-3}|&\lesssim 2^{\mathfrak{b}l} \sum_{k_2\geq 1} 2^{-3(\frac{1}{18}-s)k_2}\|f\|_{H^{-s}}\|g\|_{H^{-s}}\|h\|_{H^{-s}}\\
&\lesssim_s 2^{\mathfrak{b}l} \|f\|_{H^{-s}}\|g\|_{H^{-s}}\|h\|_{H^{-s}}.
\end{align*}

Finally, collecting the estimate for each part proves the proposition.
\end{proof}

\begin{remark}
With the generalized Sobolev improving estimate proved, one can then prove Theorem \ref{thm2} by following the argument in \cite[Section 4--6]{FGP19}, which relies on the work of \cite{LP09} and \cite{Bourgain88}, and generalizing it from a polynomial $P$ to our $\gamma$. It is routine to check it and we omit the details.

The explicit range $(8/9,1]$ of $\beta$ in Theorem \ref{thm2} is a consequence of the range $(0, 1/18)$ of $s$ in Proposition \ref{s8-prop1}. For any $8/9<\beta\leq 1$ one can apply Proposition \ref{s8-prop1} with a fixed $s=s(\beta)\in((1-\beta)/2, 1/18)$.  Note that such a choice of $s$ implies $\beta+2s>1$ which ensures that the estimations in \cite[Section 4]{FGP19} still work.

The rescaled bump function we used, $\vartheta_{l}(t)=2^{l}\vartheta(2^{l} t)$, is different from what is used in \cite{FGP19}---we have an extra factor $2^{l}$. However, this change is not essential.
\end{remark}


\appendix

\section{H\"ormander's \cite[Theorem 1.1]{Hormander73}} \label{app1}

By tracking the implicit constant in its proof, we have the following form of H\"ormander's \cite[Theorem 1.1]{Hormander73} for $L^2$.

\begin{theorem}\label{app1-1}
Let $\psi\in C_c^{\infty}(\mathbb{R}^2)$, real-valued $\varphi\in C^{\infty}(\mathbb{R}^2)$  and
\begin{equation*}
T_\lambda f(x) = \int_{\mathbb{R}} e^{i\lambda\varphi(x,y)} \psi(x,y) f(y) \,\textrm{d}y, \  f\in C_c^{\infty}(\mathbb{R}), \lambda>0.
\end{equation*}
If
\begin{equation*}
|\partial_{xy}^2 \varphi(x,y)|\geq \mathfrak{c} \ \textrm{in $\supp{\psi}$},
\end{equation*}
then
\begin{equation*}
\|T_\lambda f\|_2 \leq C C_1 |\xsupp\psi|^{1/2} \lambda^{-1/2} \|f\|_2,
\end{equation*}
where $C$ is an absolute number and
\begin{equation*}
\begin{split}
C_1=&\big\{
\mathfrak{c}^{-2}\left(\|\partial_{xx}^2\psi\|_{\infty}\|\psi\|_{\infty}+\|\partial_{x}\psi\|_{\infty}^2\right) +\mathfrak{c}^{-3}\|\psi\|_{\infty}^2\|\partial_{xxxy}^4\varphi\|_{\infty}\\
&+\mathfrak{c}^{-3}\|\partial_{x}\psi\|_{\infty}\|\psi\|_{\infty}\|\partial_{xxy}^3\varphi\|_{\infty}+
\mathfrak{c}^{-4}\|\psi\|_{\infty}^2\|\partial_{xxy}^3\varphi\|_{\infty}^2+\|\psi\|_{\infty}^2\big\}^{1/2}
\end{split}
\end{equation*}
in which we take $\|g\|_{\infty}=\esssup_{x\in \supp\psi}|g(x)|$.
\end{theorem}

Based on H\"ormander's proof it is easy to observe that the constants come from the estimate of
\begin{equation*}
K_{\lambda}(y,z):=\int_{\mathbb{R}} e^{i\lambda(\varphi(x,y)-\varphi(x,z))} \psi(x,y)\overline{\psi(x,z)} \,\textrm{d}x.
\end{equation*}
On one hand we always have the trivial bound
\begin{equation*}
|K_{\lambda}(y,z)|\leq |\xsupp\psi|\|\psi\|_{\infty}^2.
\end{equation*}
On the other hand when $\lambda|y-z|>1$ integration by parts twice gives
\begin{align*}
|K_{\lambda}(y,z)|\leq &C|\xsupp{\psi}|\left(\lambda|y-z|\right)^{-2}
\big\{\mathfrak{c}^{-2}\left(\|\partial_{xx}^2\psi\|_{\infty}\|\psi\|_{\infty}+\|\partial_{x}\psi\|_{\infty}^2\right) \\
&+\mathfrak{c}^{-3}\|\psi\|_{\infty}^2\|\partial_{xxxy}^4\varphi\|_{\infty}
+\mathfrak{c}^{-3}\|\partial_{x}\psi\|_{\infty}\|\psi\|_{\infty}\|\partial_{xxy}^3\varphi\|_{\infty}\\
&+\mathfrak{c}^{-4}\|\psi\|_{\infty}^2\|\partial_{xxy}^3\varphi\|_{\infty}^2 \big\}.
\end{align*}
Combing these two bounds yield
\begin{equation*}
\int |K_{\lambda}(y,z)|\,\textrm{d}y\leq CC_1^2|\xsupp{\psi}| \lambda^{-1}
\end{equation*}
and the same bound for $\int |K_{\lambda}(y,z)|\,\textrm{d}z$. These lead to the desired bound in the theorem.


\section{$\sigma$-uniformity} \label{app2}

For the convenience of readers we state the third author's \cite[Theorem 6.2]{Li13}.

Let $\sigma\in (0, 1]$, $\textbf{I}\subset\mathbb{R}$ be a fixed bounded interval and
$\mathfrak{U}(\textbf{I})$  a nontrivial subset of $L^2(\textbf{I})$ such that
the $L^2$-norm of every element of $\mathfrak{U}(\textbf{I})$ is uniformly
bounded by a constant. We say that a function $f\in L^2(\textbf{I})$
is \emph{$\sigma$-uniform in $\mathfrak{U}(\textbf{I})$} if
\begin{equation*}
\left|\int_{\textbf{I}} \! f(x) \overline{u(x)} \,\textrm{d}x
\right|\leq \sigma \|f\|_{L^2(\textbf{I})} \quad \textrm{for all
$u\in \mathfrak{U}(\textbf{I})$.}
\end{equation*}

\begin{lemma}\label{app2-1}
Let $\mathscr{L}$ be a bounded sublinear functional from
$L^2(\textbf{I})$ to $\mathbb{C}$, $S_\sigma$ the set of all
functions that are $\sigma$-uniform in $\mathfrak{U}(\textbf{I})$,
\begin{equation*}
A_\sigma = \sup\{|\mathscr{L}(f)|/\|f\|_{L^2(\textbf{I})} : f\in
S_\sigma, f\neq 0\}
\end{equation*}
and
\begin{equation*}
M=\sup_{u\in \mathfrak{U}(\textbf{I})}|\mathscr{L}(u)|.
\end{equation*}
Then
\begin{equation*}
\|\mathscr{L}\|\leq \max\{A_\sigma,  2\sigma^{-1} M\}.
\end{equation*}
\end{lemma}



\end{document}